\documentclass[final]{siamart220329}

\makeatletter
\def\BState{\State\hskip-\ALG@thistlm}
\makeatother

\usepackage{color}

\usepackage{enumitem}
\usepackage{stmaryrd}

\usepackage{graphicx}
\usepackage{amsfonts,mathrsfs}
\usepackage{amsmath}
\usepackage{mathtools}
\usepackage{amssymb}
\usepackage{commath}
\usepackage{bigints}
\usepackage{romanbar}
\usepackage{pgfplots}
\pgfplotsset{compat = newest}
\usepackage{float}
\usepackage{comment}
\usepackage{enumerate}
\usepackage{xspace}

\setlist[itemize]{leftmargin=*}

\newsiamthm{example}{Example}
\newsiamthm{remark}{Remark}
\newsiamthm{assumption}{Assumption}
\newsiamthm{question}{Question}
\newcommand{\ve}{\mathbf{e}}

\newcommand{\vx}{\mathbf{x}}
\newcommand{\vy}{\mathbf{y}}
\newcommand{\vz}{\mathbf{z}}
\newcommand{\vu}{\mathbf{u}}
\newcommand{\vU}{\mathbf{U}}
\newcommand{\vv}{\mathbf{v}}
\newcommand{\vw}{\mathbf{w}}

\newcommand{\vn}{\mathbf{n}}
\newcommand{\vm}{\mathbf{m}}

\newcommand{\vLambda}{\boldsymbol{\Lambda}}
\newcommand{\vF}{\mathbf{F}}
\newcommand{\vL}{\mathbf{L}}

\newcommand{\va}{\mathbf{a}}
\newcommand{\vb}{\mathbf{b}}

\newcommand{\vR}{\mathbf{R}}
\newcommand{\vA}{\mathbf{A}}
\newcommand{\vB}{\mathbf{B}}
\newcommand{\vQ}{\mathbf{Q}}

\newcommand{\jump}[1]{[ #1 ]} 
\newcommand{\Th}{\mathcal{T}_h} 
\newcommand{\Eh}{\mathcal{E}_h} 
\newcommand{\V}{\mathbb{V}} 

\newcommand{\Vh}{\mathbb{V}_h}

\newcommand{\Id}{\mathbf{I}} 
\newcommand{\I}{\normalfont{\Romanbar{1}}}

\newcommand{\divrg}{{\rm div}} 
\newcommand{\argmin}{{\rm argmin}} 

\newcommand{\wh}{\widehat}
\newcommand{\calB}{\mathcal{B}}
\newcommand{\calE}{\mathcal{E}}
\newcommand{\calI}{\mathcal{I}}

\newcommand{\tol}{\textrm{tol}}

\newcommand{\logLogSlopeTriangle}[5]
{

    \pgfplotsextra
    {
        \pgfkeysgetvalue{/pgfplots/xmin}{\xmin}
        \pgfkeysgetvalue{/pgfplots/xmax}{\xmax}
        \pgfkeysgetvalue{/pgfplots/ymin}{\ymin}
        \pgfkeysgetvalue{/pgfplots/ymax}{\ymax}

        \pgfmathsetmacro{\xArel}{#1}
        \pgfmathsetmacro{\yArel}{#3}
        \pgfmathsetmacro{\xBrel}{#1-#2}
        \pgfmathsetmacro{\yBrel}{\yArel}
        \pgfmathsetmacro{\xCrel}{\xArel}

        \pgfmathsetmacro{\lnxB}{\xmin*(1-(#1-#2))+\xmax*(#1-#2)} 
        \pgfmathsetmacro{\lnxA}{\xmin*(1-#1)+\xmax*#1} 
        \pgfmathsetmacro{\lnyA}{\ymin*(1-#3)+\ymax*#3} 
        \pgfmathsetmacro{\lnyC}{\lnyA+#4*(\lnxA-\lnxB)}
        \pgfmathsetmacro{\yCrel}{\lnyC-\ymin)/(\ymax-\ymin)} 

        \coordinate (A) at (rel axis cs:\xArel,\yArel);
        \coordinate (B) at (rel axis cs:\xBrel,\yBrel);
        \coordinate (C) at (rel axis cs:\xCrel,\yCrel);

        \draw[#5]   (A)-- node[pos=0.5,anchor=north] {1}
                    (B)-- 
                    (C)-- node[pos=0.5,anchor=west] {#4}
                    cycle;
    }
}

\title{Convergent FEM for a membrane model of liquid crystal polymer networks}
\headers{FEM for liquid crystal polymer network membranes}{Lucas Bouck, Ricardo H. Nochetto, and Shuo Yang}
\author{Lucas Bouck\thanks{Department of Mathematics, University of Maryland, College Park, Maryland 20742, USA. (\email{lbouck@umd.edu})}
\and Ricardo H. Nochetto
\thanks{Department of Mathematics and Institute for Physical Science and Technology, University of Maryland, College Park, Maryland 20742, USA. (\email{rhn@umd.edu})}
\and Shuo Yang
\thanks{Yanqi Lake Beijing Institute of Mathematical Sciences and Applications, Beijing 101408, China, and Yau Mathematical Sciences Center, Tsinghua University, Beijing 100084, China (\email{shuoyang@bimsa.cn})}
}
\date{\today}

\begin{document}
\maketitle
\begin{abstract}
We design a finite element method (FEM) for a membrane model of liquid crystal polymer networks (LCNs). This model consists of a minimization problem of a non-convex stretching energy. We discuss properties of this energy functional such as lack of weak lower semicontinuity. We devise a discretization with regularization, propose a novel iterative scheme to solve the non-convex discrete minimization problem, and prove stability of the scheme and convergence of discrete minimizers. We present numerical simulations to illustrate convergence properties of our algorithm and features of the model.
\end{abstract}

\section{Introduction}
Liquid crystals polymer networks (LCNs) are materials that can deform spontaneously upon temperature or optical actuation. In such materials, mesogens (compounds that display liquid crystal properties) are cross-linked to elastomeric polymer networks so that the nematic director (mesogen's orientation) influences the network deformation under actuation. In other words, these materials combine the features of rubber and nematic liquid crystals.
This actuation property can be widely exploited in the design of materials, such as microrobots \cite{zhao2022twisting} and biomedical devices \cite{hebert1997dynamics,li2006artificial}, to achieve non-trivial and useful shapes.

We are concerned with thin films of LCNs, which are slender materials usually mathematically modeled as $3D$ hyper-elastic bodies $\calB := \Omega\times(-t/2,t/2)$, with $\Omega\subset \mathbb{R}^2$ being a bounded Lipschitz domain and $t$ being a small thickness parameter. Classical approaches in elasticity exploit dimension reduction techniques to derive 2D models for the mid-plane deformation $\vy(\Omega)$. 

\subsection{Nematic director fields and order parameters}
Due to the nematic-elastic coupling in LCNs, director (unit length vector) fields characterize orientations of LC molecules and play a crucial role in materials deformations. 

The director field $\vm :\Omega \to\mathbb{S}^1$, so-called {\it blueprinted} director field, is pre-determined and encodes the anisotropy of mesogens on the \emph{reference} mid-plane $\Omega$. On the \emph{reference} $3D$ elastic body, $\wh{\vm}: \calB \to \mathbb{S}^2$ defines an extended blueprinted director field, and we assume it takes the form $\wh{\vm} := (\vm,0)^T$. Similarly, $\vn: \calB \to \mathbb{S}^2$ denotes the director field on the \emph{deformed} configuration. 

Depending on the strength of cross-linkings between nematic components and rubber-like polymer chains, such materials can be further classified as LCNs (sometimes it is also called liquid crystal glasses) or liquid crystals elastomers (LCEs): the former has moderate to dense crosslinks, while in the latter the density of crosslinks is low \cite{white2015programmable}. In this paper, we focus on LCNs and leave a numerical study of LCEs for future research. Mathematically, the strong coupling in LCNs is reflected in terms of director fields via a kinematic constraint \cite{ozenda2020blend}:
\begin{equation}\label{eq:kinematic-cons}
\vn:=\frac{(\nabla\vu)\wh\vm}{|(\nabla\vu)\wh\vm|},
\end{equation}
where $\vu: \calB \to\mathbb{R}^{3}$ is the $3D$ deformation. 
This implies that, in contrast to LCEs \cite{warner2007liquid,bartels2022nonlinear}, here $\vn$ is \emph{not} a free variable, and is also called a frozen director \cite{cirak2014computational}. 

Moreover, $s_0,s\in L^\infty(\Omega)$ are nematic order parameters that refer to the \emph{reference} configuration and \emph{deformed} configuration respectively. These parameters are typically constant in time and depend on temperature, but may vary in $\Omega$ if the liquid crystal polymers are actuated non-uniformly. Their physical range is $s_0,s>-1$ and $s_0,s$ are bounded away from $-1$ i.e.
\begin{equation}\label{eq:s-lower-bound}
\mathrm{essinf}_{\vx\in \Omega} s_0(\vx )>-1,\quad  \mathrm{essinf}_{\vx\in \Omega}s(\vx ) > -1 \, .
\end{equation}
The actuation parameter of the model is
\begin{equation}\label{eq:lambda}
\lambda = \lambda_{s,s_0}  = \sqrt[3]{\frac{s+1}{s_0+1}} \, .
\end{equation}
If the material is heated, then $\lambda <1$. Likewise, if cooled, then $\lambda >1$. For $s,s_0$ non-constant, the assumption on $s,s_0$ in \eqref{eq:s-lower-bound} implies that there is a constant $c_{s,s_0}$ such that $\lambda:\Omega\to \mathbb{R}$ satisfies
\begin{equation}\label{eq:lambda-bounds}
0 < c_{s,s_0}\leq \mathrm{essinf}_{\vx\in \Omega}\lambda(\vx ) \le \mathrm{esssup}_{\vx\in \Omega} \lambda(\vx ) <\infty.
\end{equation}
Therefore, equilibrium deformations of LCNs can be programmed by design of $\vm$, $s$ and $s_0$ \cite{aharoni2018universal, plucinsky2018patterning,plucinsky2016programming,plucinsky2018actuation}. We explore this feature in this and our companion paper \cite{bouck2022computation}.

\subsection{3D elastic energy: neo-classical energy}
In the context of LCNs, the starting point is the neo-classical energy density of \emph{incompressible} nematic elastomers derived by Bladon, Warner and Terentjev \cite{bladon1994deformation,warner2007liquid,warner2003thermal}. For $\vu:\calB \to \mathbb{R}^3$, the 3D energy evaluated at $\vu$ is 
\begin{equation}\label{eq:3D-energy}
E_{3D,t}[\vu] = \int_{-t/2}^{t/2}\int_{\Omega} W_{3D} \big((\vx ,z), \nabla\vu \big)d\vx dz,
\end{equation}
where the 3D energy density, known as the trace formula, is defined by
\begin{equation}\label{eq:neo-classical}
W_{3D} \big((\vx ,z), \vF\big) = \big|\vL_{\vn}^{-1/2}\vF\vL_{\vm}^{1/2}|^2 - 3,
\end{equation}
where $\vF\in\mathbb{R}^{3\times3}$ satisfies 
\begin{equation}\label{eq:incompressible}
\det \vF =1
\end{equation}
due to the assumption of incompressibility. 
The step length tensors in the reference and deformed configurations, denoted by $\vL_{\vm},\vL_\vn$, are defined as follows:
\begin{equation}\label{eq:step-length}
\vL_\vm:=(s_0+1)^{-1/3}\big(\Id_3+s_0\wh{\vm}\otimes\wh{\vm}\big),\quad \vL_{\vn}:=(s+1)^{-1/3}\big(\Id_3+s\vn\otimes\vn\big),
\end{equation}
with $\Id_3$ the identity matrix in $\mathbb{R}^3$. Note that $\vL_\vn,\vL_\vm$ are symmetric positive definite due to \eqref{eq:s-lower-bound}, so $\vL_\vn^{-1/2},\vL_\vm^{1/2}$ are well-defined.
These tensors reflect mathematically the influence of liquid crystals molecules on the network deformations when they are multiplied to deformation gradients $\vF=\nabla\vu$ in \eqref{eq:neo-classical}. Since $\vm,s,s_0$ may not be constant in $\Omega$, then the energy density in \eqref{eq:neo-classical} has a dependence on coordinates $(\vx ,z)\in\calB$. The presence of the constant $-3$ in \eqref{eq:neo-classical} ensures non-negativity of $W_{3D}$. Its role will be explained in Section \ref{sec:properties-str-energy}.

Moreover, in the specific case where $s=s_0=0$, the material becomes isotropic and the step length tensors in \eqref{eq:step-length} reduce to the identity $\Id_3$: $\vL_{\vm} = \vL_{\vn} = \Id_3$. Consequently, the trace formula simplifies to the classical incompressible neo-Hookean energy density $W^H_{3D}(\vF) = |\vF|^2- 3$ for rubber elasticity.

A $2D$ membrane model can be derived via formal asymptotics of $\lim_{t\to0} \frac{1}{t} E_{3D,t}$ after incorporating the kinematic constraint \eqref{eq:kinematic-cons}. We omit such derivation and refer to our accompanying work \cite{bouck2022computation} for a detailed discussion. The derivation is inspired by asymptotics in \cite{ozenda2020blend}, where the additional inextensibility constraint is imposed for $\vy = \vu(\cdot,0)$ and a blend of stretching and bending energy is obtained with different scaling of $t$. Moreover, \cite{cirak2014computational} derives a $2D$ energy density by taking an infimum of $W_{3D}$ over the third column of $\vF$ under the incompressibility constraint.

Next, we present the $2D$ membrane model of LCNs under consideration and state the main mathematical problem of this work.

\subsection{Problem statement: a membrane model}
The $2D$ membrane model consists of the following formal minimization problem: find $\vy^{\ast}\in H^1(\Omega; \mathbb{R}^3)$ such that
\begin{equation}\label{eq:stretching-energy}
  \vy^{\ast} \in \argmin_{\vy\in H^1(\Omega; \mathbb{R}^3)}E[\vy],
  \quad E[\vy]:= \int_\Omega W(\vx ,\nabla\vy)d\vx ,
\end{equation}
where the stretching energy density $W$ is only a function of $\vx\in \Omega$ and the first fundamental form $\I[\vy] := \nabla\vy^T\nabla\vy$ of the surface $\vy(\Omega)$ and is given by
\begin{equation}\label{eq:stretching-energy-density}
W(\vx ,\nabla\vy)= \bigg|\vL_{\vn[\vy]}^{-1/2}\big[\nabla \vy, \, \vb[\vy]\big]\vL_\vm^{1/2}\bigg|^2 - 3 \, ,
\end{equation}
where 
\begin{equation}\label{eq:shortenings}
\vn[\vy] := \frac{\nabla\vy\, \vm }{|\nabla\vy\,\vm |}, \quad J[\vy] := \det \I[\vy], \quad \vb[\vy] := \frac{\partial_1\vy\times\partial_2\vy}{J[\vy]}.
\end{equation}
Note that if $J[\vy],|\nabla\vy\,\vm |$ are bounded away from $0$, then $\int_\Omega W(\vx ,\nabla\vy)\text{} d\vx $ is finite.
We also point out that \eqref{eq:stretching-energy-density} is consistent with the stretching energy in \cite{ozenda2020blend} after additionally assuming an inextensibility constraint $J[\vy] = 1$ and incorporating the multiplicative parameter $\lambda$ and the constant $-3$.

An important warning about \eqref{eq:stretching-energy} is in order: the energy density \eqref{eq:stretching-energy-density} is not convex, which raises the question of well-posedness of \eqref{eq:stretching-energy}. In fact, we construct an explicit example in Section \ref{sec:properties-str-energy} that shows that $E$ is not weakly lower semicontinuous in $H^1(\Omega;\mathbb{R}^3)$. Therefore a direct minimization of $E$ may create or produce wrinkling and creasing, thereby leading to microstructure that we do not study in this paper. Instead, we propose a numerical regularization mechanism inspired by a bending energy that suppresses such oscillations and allows the stretching energy to drive the LCN membrane towards a preferred heterogeneous metric - the so-called target metric. For deformations $\vy\in H^1(\Omega;\mathbb{R}^3)$, this target metric condition is equivalent to $E[\vy]=0$; see Corollary \ref{cor:minimizers}. We also prove that if there is a deformation $\vy\in H^2(\Omega;\mathbb{R}^3)$ such that $E[\vy]=0$, then our algorithm computes asymptotically a (possibly different) $\vy^*\in H^2(\Omega;\mathbb{R}^3)$ such that $E[\vy^*] = 0$. This admits an important physical interpretation: $\vy\in H^2(\Omega;\mathbb{R}^3)$ with $E[\vy]=0$ is a configuration with zero membrane energy and finite bending energy. However, our algorithm is able to compute situations that fail to satisfy this regularity assumption, and yet are physically relevant. We refer to the degree $3/2$ defect later in Section \ref{sec:lc-defects} and to our companion paper \cite{bouck2022computation} for a discussion of numerous such situations.

Throughout this work, we do not impose any boundary condition so that the materials under consideration have free boundary conditions (with some abuse of language). If necessary, one can take Dirichlet boundary conditions into account with a simple modification on theories and simulations.

\subsection{Discretizations and our contributions}
There are some works in the literature about numerical analysis of methods for LCNs/LCEs. FEMs are utilized for computations of $3D$ models in \cite{conti2002soft,chung2017finite}, and in \cite{cirak2014computational} for a membrane model of nematic glasses but without a numerical analysis. In \cite{luo2012numerical,bartels2022nonlinear}, mixed FEMs (for deformations and directors) are designed for various $2D$ models of LCEs.

In this work, we propose a FEM discretization to \eqref{eq:stretching-energy}. To the best of our knowledge, this is the first numerical method with a convergence analysis for this model. We consider a continuous $\mathcal{P}_1$ Lagrange finite element approximation $\vy_h$ of the deformation. To define a discrete energy, we replace $\vy$ in \eqref{eq:stretching-energy} by $\vy_h$ and then add a \emph{regularization} term that mimics a higher order bending energy
\begin{equation*}
R_h[\vy_h] := c_rh^2\sum_{e\in\mathcal{E}_h}\frac{1}{h}\int_e \big|\jump{\nabla \vy_h} \big|^2
\end{equation*}
to deal with the non-convexity of $E$. This regularization term is a scaled $L^2$ norm of jumps $\jump{\nabla \vy_h}$ along all the edges $e\in\mathcal{E}_h$ of shape-regular meshes $\Th$, and it is critical for the proof convergence of minimizers of the discrete energy
\begin{equation}\label{eq:intro-discrete-energy}
E_h[\vy_h] := E[\vy_h] + R_h[\vy_h]
\end{equation}
in Section \ref{sec:numerical-analysis}. This proof requires the construction of a recovery sequence $\vy_h$ for $\vy\in H^2(\Omega;\mathbb{R}^3)$ with the desired energy scaling $E_h[\vy_h]\lesssim h^2$, as well as a compactness result; the energy scaling is confirmed computationally. We also extend our theory to piecewise $H^2$ deformations $\vy$ that corresponds to non-isometric origami structures. Moreover, in order to solve the discrete minimization problem, we design a nonlinear gradient flow scheme that embeds a Newton sub-iteration solving the nonlinear discrete equation at each step of the flow. This scheme is energy decreasing, and  efficient under mild conditions on the pseudo time step $\tau$.

The rest of this article is organized as follows. In Section \ref{sec:properties-str-energy}, we discuss properties of the $2D$ model, in particular the non-degeneracy of \eqref{eq:stretching-energy-density} and their related consequences. We also construct an explicit example illustrating the lack of weak lower semicontinuity of the stretching energy and discuss our strategy to deal with it. In Section \ref{sec:discrete-min-pb}, we introduce the discrete version of \eqref{eq:stretching-energy} and the nonlinear gradient flow scheme to solve it. In Section \ref{sec:numerical-analysis}, we present a convergence analysis of discrete minimizers, our main contribution, and in turn show useful technical tools, the construction of a recovery sequence, a compactness argument and an extension to piecewise $H^2$ deformations. 
The main technical difficulty in proving the desired energy scaling is that $W(\vx,\nabla\vy(\vx))\to\infty$ as $J[\vy(\vx)]\to0$. To avoid the singularity, we need a pointwise lower bound of $\det\nabla\vy_h^T\nabla\vy_h$ uniform in $h$. Since the space $H^2(\Omega)$ has the same Sobolev number as $W^{1,\infty}(\Omega)$ for $\Omega\subset\mathbb{R}^2$, whence it does not embed compactly, and $\vy\mapsto J[\vy]$ is not concave, interpolating directly or after convolution may be problematic. To overcome this challenge in Section \ref{S:prelim-energy-scaling}, we employ a Lusin truncation of Sobolev functions, motivated by \cite{friesecke2002theorem}, which provides both the desired lower bound of $J[\vy_h]$ and first order convergence in $H^1(\Omega)$. This approach may be useful for other critical nonlinear problems. We note that a Lusin truncation of $W^{1,1}_0(\Omega)$ functions has been used in numerical analysis for incompressible fluids with an implicit constitutive law \cite{diening2013finite}.
Finally, we conclude in Section \ref{sec:simulations} with numerical simulations, including experiments with origami shapes, a quantitative study for the convergence of the proposed method, and an example with a stable defect of degree $3/2$ that goes beyond the theory.

\section{Properties of the stretching energy}\label{sec:properties-str-energy}

This section is dedicated to proving some properties of the stretching energy \eqref{eq:stretching-energy}, which will be useful later in Section \ref{sec:numerical-analysis}. 
In view of definition \eqref{eq:step-length} for step length tensors, we first observe that $\vL_\vm$ can be equivalently expressed  as follows in the orthonormal basis $\{\hat{\vm}, \hat{\vm}_\perp, \ve_3\}$:
\begin{equation}\label{eq:Lm-rewrite}
\vL_\vm = (s_0+1)^{2/3}\wh{\vm}\otimes \wh{\vm} +(s_0+1)^{-1/3}\wh{\vm}_\perp\otimes \wh{\vm}_\perp+ (s_0+1)^{-1/3}\ve_3\otimes \ve_3.
\end{equation}
Likewise $\vL_\vn$ may be expressed in the basis $\{\vn,\vv_1,\vv_2\}$ for orthonormal vectors $(\vv_1,\vv_2)$ spanning the space orthogonal to $\vn$:
\begin{equation}\label{eq:Ln-rewrite}
\vL_\vn = (s+1)^{2/3}\vn\otimes \vn +(s+1)^{-1/3}\vv_1\otimes \vv_1+ (s+1)^{-1/3}\vv_2\otimes \vv_2.
\end{equation}
The assumptions \eqref{eq:s-lower-bound} together with $s_0,s\in L^\infty(\Omega)$ imply that the eigenvalues of $\vL_\vm$ and $\vL_\vn$ are bounded away from $0$ and $\infty$, and $\vL_\vm, \vL_\vn$ are thus invertible.
Moreover, \eqref{eq:Ln-rewrite} provides an explicit inverse for $\vL_\vn$
\begin{equation*}
\vL_\vn^{-1} = (s+1)^{-2/3}\vn\otimes \vn +(s+1)^{1/3}\vv_1\otimes \vv_1+ (s+1)^{1/3}\vv_2\otimes \vv_2 \, ,
\end{equation*}
or equivalently
\begin{equation}\label{eq:Ln-inv}
\vL_\vn^{-1} = (s+1)^{1/3}\left(\Id_3 -\frac{s}{s+1}\vn\otimes \vn\right).
\end{equation}

\subsection{Coercivity}\label{sec:str-energy-coercivity}
In this subsection we show coercivity of the stretching energy.

\begin{proposition}[coercivity]\label{prop:coercivity}
There exists $C(s,s_0)>0$ such that the stretching energy $E$ defined in \eqref{eq:stretching-energy} satisfies
\begin{equation}\label{eq:str-coercivity}
  C(s,s_0)\left(\Vert\nabla \vy\Vert_{L^2(\Omega;\mathbb{R}^{3\times2})}^2+\Vert J[\vy]^{-1/2}\Vert_{L^2(\Omega)}^2\right)-3|\Omega| \leq E[\vy]
  \quad\forall \, \vy\in H^1(\Omega;\mathbb{R}^3).
  \end{equation}
\end{proposition}
\begin{proof}
Recall the expressions \eqref{eq:stretching-energy} and \eqref{eq:stretching-energy-density},
\begin{equation*}
E[\vy] =\int_\Omega \left(\left|\vL_\vn^{-1/2}[\nabla\vy,\vb]\vL_\vm^{1/2}\right|^2- 3\right)\text{ }d\vx .
\end{equation*}
where $\vn = \vn[\vy]$ and $\vb = \vb[\vy]$ as in \eqref{eq:shortenings}. 
We now invoke an elementary result for any matrix $\vA\in \mathbb{R}^{d\times d}$ and a symmetric positive definite (SPD) matrix $\vB\in \mathbb{R}^{d\times d}$: $\left|\vA\vB\right|^2 \geq \lambda_{min}(\vB)^2|\vA|^2$ and $\left|\vB\vA\right|^2 \geq \lambda_{min}(\vB)^2|\vA|^2$ where $\lambda_{min}(\vB) = \min_{1\leq j\leq d}\{\lambda_j(\vB)\}>0$. 

These properties allow us to write the lower bound
\begin{equation*}
E[\vy] \geq \int_\Omega \left(\frac{\lambda_{min}(\vL_{\vm})}{\lambda_{max}(\vL_{\vn})}\left|[\nabla\vy,\vb]\right|^2- 3\right)\text{ }d\vx .
\end{equation*}
In view of the forms of $\vL_\vm,\vL_\vn$ in \eqref{eq:Lm-rewrite} and \eqref{eq:Ln-rewrite}, their eigenvalues are explicit, namely $(s_0+1)^{2/3},(s_0+1)^{1/3}$ and $(s+1)^{2/3}, (s+1)^{2/3}$ respectively. Recalling the assumptions on $s,s_0$ in \eqref{eq:s-lower-bound}, we have that there is a constant $C(s,s_0)>0$ such that $\frac{\lambda_{min}(\vL_{\vm})}{\lambda_{max}(\vL_{\vn})} \geq C(s,s_0)$ for a.e.\ $\vx\in \Omega$. Thus,
\begin{equation*}
E[\vy] \geq C(s,s_0)\int_\Omega\left|[\nabla\vy,\vb]\right|^2\text{ }d\vx - 3|\Omega|.
\end{equation*}
We observe that
\begin{equation}\label{eq:det-normal}
J[\vy] = |\partial_1\vy|^2|\partial_2\vy|^2 - (\partial_1\vy\cdot\partial_2\vy)^2 = |\partial_1\vy\times\partial_2\vy|^2
\end{equation}
due to the cross product identity $|\va\times\vb|^2 = |\va|^2|\vb|^2 - (\va\cdot\vb)^2$ for vectors $\va,\vb\in \mathbb{R}^3$.
As a consequence of the above formula with the definition of $\vb[\vy]$ in \eqref{eq:shortenings}, we have that $|\vb|^2 = J[\vy]^{-1}$. Realizing this fact completes the proof.
\end{proof}

\subsection{Non-degeneracy and global minimizers}

Throughout this subsection, we prove some properties of $W$ that involve matrix properties.
If the second argument of $W$ is a generic matrix $\vF\in\mathbb{R}^{3\times 2}$ instead of $\nabla\vy$, we then follow \eqref{eq:shortenings} and define $\vb(\vF), J(\vF), \vn(\vF)$ as
\begin{equation}\label{eq:shortenings-F}
 \vn(\vF)  := \frac{\vF \, \vm }{|\vF\,\vm |}, \quad J(\vF)  := \det (\vF^T\vF) , \quad  \vb(\vF) := \frac{\vF_1\times\vF_2}{J(\vF) },
\end{equation}
where $\vF_1,\vF_2$ denotes the first and second columns of $\vF$ respectively.

We recall that the energy density in \eqref{eq:stretching-energy-density} can be rewritten as a neo-Hookean energy density:
\begin{equation}\label{eq:neo-Hookean}
W^H_{3D}(\vF) = |\vF|^2- 3
\end{equation}
where $\vF\in \mathbb{R}^{3\times3}$ (with a slight abuse of notation).
In particular, we want to exploit the relation between the neo-Hookean structure of the 2D stretching energy \eqref{eq:stretching-energy} to derive the non-degeneracy and properties of global minimizers. We also stress the importance of \eqref{eq:neo-Hookean} because it is critical for the energy scaling argument in Proposition \ref{prop:energy-scaling}. To see the non-negativity of \eqref{eq:neo-Hookean}, we first observe that a basic linear algebra argument exploiting the eigenvalues of $\vF^T\vF$ yields $W^H_{3D}(\vF) = |\vF|^2- 3\ge0$ provided $\det \vF = 1$. 

More precisely, $W^H_{3D}(\vF)$ is non-degenerate in the sense that it is bounded from below by $\mathrm{dist}(\vF,SO(3))^2 := \inf_{\vR\in SO(3)}|\vF - \vR|^2$. We now state and prove lower and upper bounds for $W^H_{3D}(\vF)$. The former can also be found in \cite[Proposition A.3]{plucinsky2018actuation}. The latter will be used in the numerical analysis in Lemma \ref{lem:expansion}.
\begin{proposition}[bounds for $W^H_{3D}(\vF)$]\label{prop:nondegen-neo-Hookean}
Let $\vF\in \mathbb{R}^{3\times3}$ satisfy $\det \vF = 1$. Then,
\begin{equation}\label{eq:nondegen-neo-Hookean}
\mathrm{dist}\big(\vF,SO(3)\big)^2 \leq |\vF|^2 - 3 \leq 3\;\mathrm{dist}\big(\vF,SO(3)\big)^2 \, .
\end{equation}
\end{proposition}
\begin{proof}
  Let $\vF\in \mathbb{R}^{3\times3}$ be such that $\det \vF = 1$. We first use the polar decomposition, $\vF = \vR\vU$ for $\vU$ (SPD) and $\vR\in SO(3)$, to write $|\vF|^2~ - ~3 = |\vR\vU|^2 - 3 = |\vU|^2 - 3$, and $\mathrm{dist}(\vR\vU,SO(3))^2 =\mathrm{dist}(\vU,SO(3))^2$.

1. {\it Lower bound}: It is thus sufficient to prove
\begin{equation*}
|\vU|^2 - 3\geq \mathrm{dist}(\vU,SO(3))^2.
\end{equation*}
Since $\vU$ is SPD there exists $\vQ\in SO(3)$ such that $\vU=\vQ^T\vLambda \vQ$ with $\vLambda$ a diagonal matrix with the eigenvalues $\lambda_1,\lambda_2,\lambda_3>0$ of $\vU$. Moreover, $\det\vU=1$ yields $\lambda_3= \frac{1}{\lambda_1\lambda_2}$, and $|\vU|=|\vLambda|$ implies
\begin{equation*}
|\vU|^2 - 3 = \lambda_1^2+\lambda_2^2 +\frac{1}{\lambda_1^2\lambda_2^2} - 3.
\end{equation*}
On the other hand, $\mathrm{dist}(\vU,SO(3)) = |\vU - \Id_3|$ because $|\vU-\vR|=|\vLambda- \vQ\vR\vQ^T|$ with $\vR\in SO(3)$ is minimized by $\vQ\vR\vQ^T=\Id_3$, whence $\vR=\Id_3$. Consequently,
\begin{align}
\mathrm{dist}(\vU,SO(3))^2 &= (\lambda_1-1)^2+(\lambda_2-1)^2 +\left(\frac{1}{\lambda_1\lambda_2}-1\right)^2 \label{eq:key-equality1}\\
& = |\vU|^2 - 3+ 2\left(3 - \lambda_1-\lambda_2-\frac{1}{\lambda_1\lambda_2}\right). \label{eq:key-equality2}
\end{align}
A basic calculus argument gives $\sup_{\lambda_1,\lambda_2>0}\left(3 - \lambda_1-\lambda_2-\frac{1}{\lambda_1\lambda_2}\right)\leq0$, whence
\begin{equation*}
\mathrm{dist}(\vU,SO(3))^2 \leq |\vU|^2 - 3,
\end{equation*}
and the lower bound is proved.

\medskip
2. {\it Upper bound}: In view of \eqref{eq:key-equality1} and \eqref{eq:key-equality2} it suffices to prove
\[
(\lambda_1-1)^2+(\lambda_2-1)^2 + \left(\frac{1}{\lambda_1\lambda_2}-1\right)^2\ge \lambda_1+\lambda_2+\frac{1}{\lambda_1\lambda_2}-3. 
\]
Without loss of generality, let us assume $\lambda_3 = \frac{1}{\lambda_1\lambda_2} \geq 1$ and write
\[
\lambda_1+\lambda_2 + \frac{1}{\lambda_1\lambda_2} -3
= \Big( \frac{1}{\lambda_1\lambda_2} - 1\Big) - \Big(1 - \lambda_1\lambda_2 \Big)
+ \lambda_1 + \lambda_2 - \lambda_1\lambda_2 - 1.
\]
The first two terms satisfy the following relation
\[
\Big( \frac{1}{\lambda_1\lambda_2} - 1\Big) - \Big(1 - \lambda_1\lambda_2 \Big)
= \lambda_1\lambda_2 \Big( \frac{1}{\lambda_1\lambda_2} - 1\Big)^2 \le \Big( \frac{1}{\lambda_1\lambda_2} - 1\Big)^2
\]
because $\lambda_1\lambda_2\le1$. The remaining terms, instead, obey the relation
\[
\lambda_1 + \lambda_2 - \lambda_1\lambda_2 - 1 = - (\lambda_1-1)(\lambda_2-1) \le (\lambda_1-1)^2 + (\lambda_2-1)^2
\]
by virtue of Young's inequality. This proves the desired upper bound. 
\end{proof}

The next nondegeneracy estimate will be useful for the subsequent discussion.

\begin{corollary}[nondegeneracy of stretching energy]\label{cor:stretching-nondegeneracy}
The stretching energy density $W(\vx ,\vF) = \left| \vL_\vn^{-1/2}[\vF,\,\vb]\vL_\vm^{1/2}\right|^2 - 3$ satisfies
\begin{equation}
W(\vx ,\vF)\geq  \mathrm{dist} \big(\vL_{\vn}^{-1/2}[\vF, \; \vb]\vL_{\vm}^{1/2},SO(3) \big)^2, 
\end{equation}
for all $\vF\in \mathbb{R}^{3\times2}$ such that $\mathrm{rank}(\vF)=2$, $\vn = \vn(\vF)$, and $\vb = \vb(\vF)$ as defined in \eqref{eq:shortenings-F}.
\end{corollary}
\begin{proof}
Arguing similarly to the derivation \eqref{eq:det-normal} in the proof of Proposition \ref{prop:coercivity} (coercivity), we use the cross product identity $|\va\times\vb|^2 = |\va|^2|\vb|^2 - (\va\cdot\vb)^2$ for vectors $\va,\vb\in \mathbb{R}^3$ to deduce
\begin{equation*}
J(\vF) = |\vF_1|^2|\vF_2|^2 - (\vF_1\cdot\vF_2)^2 = |\vF_1\times\vF_2|^2
\end{equation*}
According to \eqref{eq:shortenings-F} and the above equality, we have $\vb(\vF) = \frac{\vF_1\times\vF_2}{|\vF_1\times\vF_2|^2}$. With this form of $\vb(\vF)$, we observe
\[
\det [\vF,\,\vb] =\det [\vF_1,\,\vF_2,\,\vb] = (\vF_1\times\vF_2)\cdot\vb = 1.
\]
Since $\det\vL_\vm = \det\vL_\vn^{-1} = 1$ in view of \eqref{eq:Lm-rewrite} and \eqref{eq:Ln-rewrite}, we can simply apply Proposition \ref{prop:nondegen-neo-Hookean} (bounds for $W^H_{3D}(\vF)$).
\end{proof}

\begin{remark}[special rotations]\label{rmk:str-SO3}
\rm
An important by-product of Corollary \ref{cor:stretching-nondegeneracy} is that any solution $\vy\in H^1(\Omega;\mathbb{R}^3)$ of $E[\vy] = 0$ must satisfy the pointwise relation
\[
\vL_{\vn}^{-1/2}[\nabla\vy, \; \vb]\vL_{\vm}^{1/2} \in SO(3)
\]
a.e. in $\Omega$ where $\vb = \frac{\partial_1\vy\times \partial_2\vy}{|\partial_1\vy\times \partial_2\vy|^2}$ is a scaled normal. This observation will turn out to be useful later in the proof of Proposition \ref{prop:energy-scaling}.
\end{remark}

The next proposition states a known fact in the physics literature \cite{modes2011gaussian, warner2007liquid}, namely that minimizing the stretching energy pointwise is equivalent to satisfying a metric condition; we prove and discuss this result in \cite{bouck2022computation}. We refer to \cite[Appendix A]{plucinsky2018actuation} for a similar result, but for a related 3 dimensional model. We also refer to \cite[Theorem 1.13]{plucinsky2018actuation} for result showing that the metric condition arises from a vanishing thickness limit of a 3D energy at the bending energy scaling.

\begin{proposition}[target metric]\label{prop:target-metric} 
The stretching energy density $W(\vx ,\vF) = 0$ if and only if $\I(\vF)$ satisfies the metric condition $\I(\vF) = g$,  with $g\in\mathbb{R}^{2\times2}$ given by
\begin{equation}\label{eq:target-metric}
g = \lambda ^2 \vm \otimes\vm  +\lambda ^{-1}\vm _\perp\otimes\vm _\perp,
\end{equation}
and $\lambda$ defined in \eqref{eq:lambda} .
\end{proposition}

A deformation $\vy\in H^1(\Omega;\mathbb{R}^3)$ is an {\it $H^1$} isometric immersion of $g$ provided $\I(\vy)=g$ a.e. in $\Omega$. Therefore, Proposition \ref{prop:target-metric} establishes an equivalence between isometric immersions and minimizers of $E$. We make this explicit next.

\begin{corollary}[immersions of $g$ are minimizers with vanishing energy]\label{cor:minimizers}
A deformation $\vy\in H^1(\Omega;\mathbb{R}^3)$ satisfies $\I[\vy]=\nabla\vy^T\nabla\vy=g$ a.e. in $\Omega$, if and only if $\vy$ minimizes $E$ over $H^1(\Omega;\mathbb{R}^3)$ with $E[\vy]=0$.
\end{corollary}

Therefore, 
if the given data $\vm,s,s_0$ are such that $g$ admits an $H^1$ isometric immersion,
global minimizers of $E[\vy]$ over $H^1(\Omega;\mathbb{R}^3)$ are guaranteed to exist; otherwise, $E[\vy]$ may not vanish over $H^1(\Omega;\mathbb{R}^3)$. On the other hand, minimizers of $E[\vy]$ might not be unique, because $g$ could have many isometric immersions in general. From another point of view, this issue is also related to lack of convexity for $E[\vy]$. 

\subsection{Lack of weak lower semicontinuity in $H^1$}\label{sec:lack-of-convexity}
The lack of convexity of the stretching energy \eqref{eq:stretching-energy-density} translates into lack of weak lower semicontinuity of \eqref{eq:stretching-energy} and prevents one from using the direct method of calculus of variations to prove the existence of minimizers, and is also responsible for serious computational challenges.

To stress the importance of convexity or lack there-of, we present a modification of a classical 1D example known as the Bolza example \cite[Example 4.8]{dacorogna2007direct}; see also \cite[Example 2.1]{ball1998calculus}. We next extend this situation to 2D.

\begin{example}\label{ex:1D-sawtooth}
\rm We consider the double well energy defined on $W^{1,4}_0((0,1))$
\begin{equation}\label{eq:1d-double-well}
E_{1D}[u] = \int_0^1 \left((u')^2-1\right)^2+cu^2dx,
\end{equation}
with some nonnegative $c$, and define a sequence of sawtooth functions starting with 
\begin{equation*}
u_1(x) = \begin{cases} 
x,& \quad x<\frac{1}{2}\\
1-x, & \quad x\geq \frac{1}{2}.
\end{cases} 
\end{equation*}
To construct $u_2$, we subdivide the intervals $[0,1/2]$ and $[1/2,1]$ into $[0,1/4]$, $[1/4,1/2]$ and $[1/2,3/4]$, $[3/4,1]$ and then alternating the derivative between $\pm1$ on the 4 subintervals. The function $u_2$ is a sawtooth with derivative of $\pm1$ and maximum height $\frac{1}{4}$. Given $u_n$, we do the same subdividing procedure to get a $u_{n+1}$ to get a sawtooth of height $\frac{1}{2^{n+1}}$. The resulting sequence consists of $u_n$ that satisfy $|u_n'(x)|=1$. The first few elements are plotted in Figure \ref{fig:1d-sawtooth}. The sequence $u_n\rightharpoonup 0$ in $W^{1,4}((0,1))$, but
\begin{equation*}
0 = \lim_{n\to \infty} E_{1D}[u_n] < E_{1D}[0] = 1
\end{equation*}
Thus, the energy $E_{1D}$ is not weakly lower semicontinuous on $W^{1,4}$, and if $c>0$ the direct method of the calculus of variations would fail to provide the existence of a minimizer. If $c=0$, then any $u_n$ is a minimizer to $E_{1D}$ over $W^{1,4}$.
\begin{figure}[H]
\caption{Example \ref{ex:1D-sawtooth}: First four elements $u_n$ of the minimizing sequence of \eqref{eq:1d-double-well}.}\label{fig:1d-sawtooth}
\begin{center}
\includegraphics[width=.5\textwidth]{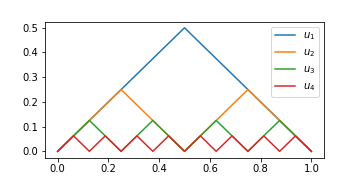}
\end{center}
\end{figure}
\end{example}

On the discrete level, the above example is also important, because the lack of convexity means that a standard weak compactness result in $H^1$ will not be enough to prove convergence of minimizers. We shall see that $E$ is not weakly lower semicontinuous in $H^1(\Omega;\mathbb{R}^3)$.
To illustrate this point, we present an example of some minimizers to $E$ that extends Example \ref{ex:1D-sawtooth} to 2D. The first element of the sequence of minimizers is a pyramid from \cite{modes2011blueprinting}. We later display several pyramid configurations in Section \ref{sec:pyramids} computed with our FEM.
\begin{example}\label{ex:pyramid-sawtooth}
\rm Let $\Omega = [-1,1]^2$ and let $\vm$ be the blueprinted director field depicted in Figure \ref{fig:pyramid-sawtooth}(a) and let $\vy_1$ be the solution in Figure \ref{fig:pyramid-sawtooth}(b) with $\lambda<1$. The surface $\vy_1(\Omega)$ is a square pyramid with base width $2\lambda$ and height $\sqrt{\lambda^{-1}-\lambda^{2}}$, and first fundamental form $\I[\vy]=g$ with target metric $g$ given by \eqref{eq:target-metric}. We can mimic the subdivision procedure of Example \ref{ex:1D-sawtooth} to produce a sequence $\vy_n$ such that $\I[\vy_n] = g$, and $\vy_n\rightharpoonup \vy^{\ast}$ in $H^1(\Omega;\mathbb{R}^3)$, where $\vy^{\ast}(\vx ) = (\lambda x_1,\lambda x_2,0)$. The first three elements of the sequence are displayed in Figure \ref{fig:pyramid-sawtooth} (b-d). Since $\I[\vy_n] = g$, we deduce $E[\vy_n]=0$ for all $n$, according to Corollary \ref{cor:minimizers} (immersions of $g$ are minimizers). Moreover, $\I[\vy^{\ast}] = \lambda^2\Id_2\neq g$ a.e.\ in $\Omega$ because $\lambda\neq1$. Inserting $\nabla \vy^{\ast}$ into $W$ yields $W(\vx ,\nabla\vy^{\ast})>0$ a.e.\ in $\Omega$ due to Proposition \ref{prop:target-metric} (target metric), whence
\begin{equation*}
\liminf_{n\to\infty} E[\vy_n] = 0 < E[\vy^{\ast}] .
\end{equation*}
We thus conclude that $E$ is not weakly lower semicontinuous in $H^1(\Omega;\mathbb{R}^3)$.
\end{example}

\begin{figure}[H]
  \caption{Example \ref{ex:pyramid-sawtooth}: Blueprinted director field $\vm$ and first three elements $\vy_n(\Omega)$ of a minimizing sequence
  with foldings on dyadic squares concentric with $\Omega$.}\label{fig:pyramid-sawtooth}
  \begin{minipage}{.15\textwidth}
\begin{center}
\begin{tikzpicture}[scale=1]
    \draw (1,-1) -- (1,1) -- (-1,1) -- (-1,-1) -- (1,-1);
    \draw (-1,-1) -- (1,1);
    \draw (-1,1) -- (1,-1);
    \draw [-to](.75,-.25) -- (.75,.25);
    \draw [-to](.25,.75) -- (-.25,.75);
    \draw [-to](-.75,.25) -- (-.75,-.25);
    \draw [-to](-.25,-.75) -- (.25,-.75);
\end{tikzpicture}
\end{center}
\end{minipage}\hfill
\begin{minipage}{.28\textwidth}
\begin{tikzpicture}[scale=.42]
\begin{axis}[view = {30}{50},zmin=0,zmax=1.32]
 
\addplot3 [
    domain =-.5:.5,
    domain y = -.5:.5,
    samples = 32,
    samples y = 32,
    surf] {1.32-  2.64*max( abs(x),abs(y))};
 
\end{axis}
 
\end{tikzpicture}
\end{minipage}\hfill
\begin{minipage}{.28\textwidth}
\begin{tikzpicture}[scale=.42]
 
\begin{axis}[view = {30}{50},zmin=0,zmax=1.32]
 
\addplot3 [
    domain =-.5:.5,
    domain y = -.5:.5,
    samples = 32,
    samples y = 32,
    surf] {min(2.64*max( abs(x),abs(y) ),1.32-  2.64*max( abs(x),abs(y)))};
 
\end{axis}
 
\end{tikzpicture}
\end{minipage}\hfill
    \begin{minipage}{.28\textwidth}
        \begin{tikzpicture}[scale=.42]
            
            \begin{axis}[view = {30}{50},zmin=0,zmax=1.32]
            
                \addplot3 [
    domain =-.5:.5,
    domain y = -.5:.5,
    samples = 32,
    samples y = 32,
                surf] {1.32*max(min(.5-max( 2*abs(x),2*abs(y) ),max( 2*abs(x),2*abs(y) )),min(1-max( 2*abs(x),2*abs(y) ), max( 2*abs(x),2*abs(y) )-.5))};
                
            \end{axis}
            
        \end{tikzpicture}
    \end{minipage}
\end{figure}
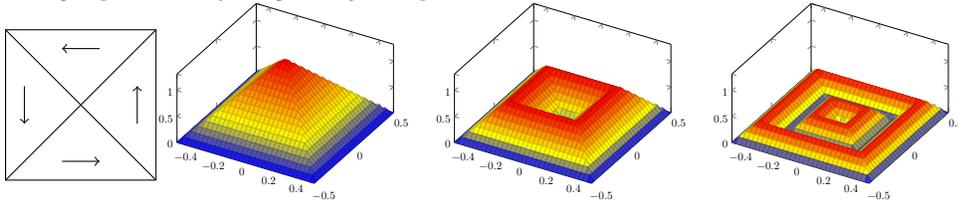
%

Additionally, we note that the relevant convexity notion for $W(\vx,\cdot)$ is quasiconvexity. We refer to the book \cite{dacorogna2007direct} for background on this topic.

There are numerous strategies to treat non-convexity in numerical methods for nonconvex energies. For an introduction on discretizations for nonconvex variational problems, we refer to \cite[Chapter 9]{bartels2015numerical}. For this paper, we choose to regularize the stretching energy with the expectation that regularization provides a stronger compactness result to get around the issue of lack of weak lower semicontinuity. We note that this strategy has been used before in the study of LCE/LCNs \cite{cirak2014computational}.
The model of \cite{cirak2014computational} utilizes the regularized energy
\begin{equation*}
\int_\Omega W(\vx ,\nabla \vy) + \varepsilon |\divrg \nabla \vy|^2,
\end{equation*}
where $\varepsilon>0$ is a positive fixed constant. This is a dimensionally reduced model from the 3D model of \cite{bhattacharya1999theory}, which incorporates a Hessian term to the energy.

We are interested in the membrane model and would like to recover the target metric in the limit. We consider the regularized energy
\begin{equation}\label{eq:regularization}  
E_\varepsilon[\vy] = E[\vy] + \varepsilon \int_\Omega \big| D^2\vy \big|^2,
\end{equation}
where $\varepsilon$ scales likes $h^2$. One may view this as analogous to a higher order bending term. This kind of energy blending is studied by \cite{ozenda2020blend}. The perspective of the regularization acting like a higher order bending term motivates the choice of $\varepsilon\approx h^2$. In fact, the best energy scaling one can expect of $E$ is $h^2$ due to the $H^2$ regularity of zero energy states of $E$ under Assumption \ref{as:H2-immersibility} (regularity), and the regularization term balances with the stretching energy scaling. One may also consider a more physical bending energy. A recent example of a bending theory for LCEs is \cite{bartels2022nonlinear}.

Another approach would be to compute minimizers of an effective energy, whose energy density is the quasiconvex envelope of $W_{str}$. In LCNs/LCEs, the authors of \cite{cesana2015effective} explicitly compute a quasiconvex energy density for a membrane energy of LCEs, and the authors of \cite{desimone2002macroscopic} compute minimizers of a relaxed 3D energy for LCEs where the effective energy is known. If the quasiconvex envelope is unknown, then one can approximate the envelope, as done for LCNs/LCEs in \cite{conti2018adaptive}. Computing the quasiconvex envelope of $W_{str}$ is outside the scope of this paper.

\section{The discrete minimization problem}
\label{sec:discrete-min-pb}
This section introduces the discrete minimization and proposes a discrete gradient flow as a solver.

\subsection{Discrete energies}
Let $\mathcal{T}_h$ be a shape regular sequence of meshes with maximum mesh size $h$. We denote by $\mathcal{E}_h$ the set of interior edges to the mesh, and by $\mathcal{N}_h$ the set of nodes of the mesh. The space for discrete deformations consists of continuous piecewise linear functions:
\begin{equation}
\mathbb{V}_h := \{ \vy_h\in C^0(\Omega;\mathbb{R}^3) : \quad \vy_h|_T \in \mathcal{P}_1 \quad\forall \; T\in \mathcal{T}_h\}.
\end{equation}
We propose the regularized discrete energy $E_h: \mathbb{V}_h\to \mathbb{R}$ defined by
\begin{equation}\label{eq:discrete-energy}
E_h[\vy_h] =\int_\Omega W(\vx ,\nabla\vy_h) \, d\vx  + R_h[\vy_h]
\end{equation}
where the regularization term $R_h[\vy_h] := c_r h^2 |\vy_h|^2_{H^2_h(\Omega;\mathbb{R}^3)}$ is a rescaling of the DG discrete $H^2$-seminorm for continuous piecewise linear functions:
\begin{equation}\label{eq:discrete-H2}
|\vy_h|_{H^2_h(\Omega;\mathbb{R}^3)}^2 := \sum_{e\in\mathcal{E}_h}\frac{1}{h}\int_e |\jump{\nabla \vy_h}|^2,
\end{equation}
and $c_r:\Omega\to\mathbb{R}^+$ is a non-negative \emph{regularization} parameter of our choice.
The notation $\jump{\nabla \vy_h}$ denotes the jump of ${\nabla \vy_h}$ across edges $e\in\calE_h$
\begin{equation}
\jump{\nabla \vy_h} \big|_e = \nabla \vy_h^+ - \nabla \vy_h^-,
\end{equation}
where $\nabla \vy_h^{\pm}(\vx ):=\lim\limits_{s\to0^+}\nabla \vy_h(\vx \pm s\vn_e)$ and $\vn_e$ is a unit normal vector to $e$ (the choice of its direction is arbitrary but fixed).
To justify that \eqref{eq:discrete-H2} is indeed a discrete $H^2$-seminorm, we argue heuristically as follows.
Since $\vy_h$ is elementwise affine, we view $H_h[\vy_h]|_e := \frac{\jump{\nabla \vy_h}|_e}{h}$ as a finite difference Hessian of $\vy_h$ across $e$. If one extends the definition of $H_h[\vy_h]|_e$ to elements $T\in\Th$ as a constant, namely, $H_h[\vy_h]^2|_T = \sum_{e\subset \partial T}\frac{1}{h}\int_e \frac{|\jump{\nabla\vy_h}|^2}{h^2}$, it satisfies the natural scaling $\int_T |H_h[\vy_h]|^2\approx h\int_e|H_h[\vy_h]|^2$ and results in
\begin{align*}
  |\vy_h|_{H^2_h(\Omega;\mathbb{R}^3)}^2 &= \sum_{e\in\mathcal{E}_h}\frac{1}{h}\int_e |\jump{\nabla \vy_h}|^2 = \sum_{e\in\mathcal{E}_h}h\int_e \big|H_h[\vy_h]\big|^2
  \\
  &\approx \sum_{T\in\mathcal{T}_h}\int_T \big|H_h[\vy_h]\big|^2 = \int_\Omega  \big|H_h[\vy_h]\big|^2.
\end{align*}

The regularization term $h^2\int_\Omega  \big|H_h[\vy_h]\big|^2$ thus mimics a higher order bending energy \eqref{eq:regularization} where $h$ is proportional to the thickness of a thin 3D body.

Our next task is to solve the discrete counterpart of \eqref{eq:stretching-energy}, namely
\begin{equation}\label{eq:minimization-pb-discrete}
\vy^{\ast}_h=\argmin_{\vy_h\in\mathbb{V}_h}E_h[\vy_h].
\end{equation} 
According to the discussions in Section \ref{sec:lack-of-convexity}, we can also expect lack of convexity and weak lower semicontinuity in the first term of $E_h[\vy_h]$ in \eqref{eq:discrete-energy}. These features brings the main difficulty to solve the discrete minimization problem \eqref{eq:minimization-pb-discrete} and to analyze convergence of $\vy_h^*$ towards a minimizer $\vy^*$ of \eqref{eq:stretching-energy}. These topics are discussed in Sections \ref{sec:nonlinear-gf} and \ref{sec:numerical-analysis}.

\subsection{Minimizing the discrete energy: nonlinear gradient flow}\label{sec:nonlinear-gf}

We design a nonlinear discrete gradient flow to find a solution to \eqref{eq:minimization-pb-discrete} in this subsection. 
Given $\vy^{i}_h~\in~ \V_h$ for any $i\ge0$, at each step of the iterative scheme we find $\vy^{i+1}_h\in \V_h$ that solves the minimization problem for the augmented functional $L^i_h$
\begin{align}\label{eq:gf-energy}
  \vy_h^{i+1} = \argmin_{\vy_h\in\V_h}L^i_h[\vy_h],
  \quad
  L^i_h[\vy_h] :=\frac{1}{2\tau}\|\vy_h-\vy^i_h\|^2_{H^1(\Omega;\mathbb{R}^3)}+E_h[\vy_h].
\end{align}
The first term in $L^i$ dictates the flow metric and penalizes the deviation of $\vy_h^{i+1}$ from $\vy_h^i$ provided the pseudo time-step $\tau>0$ is small. Therefore, \eqref{eq:gf-energy} is a mechanism to minimize $E_h[\vy_h]$ within an $H^1$-neighborhood of $\vy_h^i$. The choice of $H^1(\Omega;\mathbb{R}^3)$-metric is made for convenience, but has important consequences for the stability of \eqref{eq:gf-energy}.

Calculating the first order variation $\delta L^i_h[\vy_h](\vv_h)$ of $L^i_h[\vy_h]$ in the direction $\vv_h\in\V_h$, we obtain the weak form of the Euler-Lagrange equation for $L^i_h[\vy_h]$
\begin{equation}\label{eq:gf-onestep}
  \delta L^i_h[\vy_h](\vv_h)=\frac{1}{\tau}(\vy_h,\vv_h)_{H^1(\Omega;\mathbb{R}^3)}+\delta E_h[\vy_h](\vv_h)-F_i(\vv_h)=0
  \quad\forall \, \vv_h\in\V_h,
\end{equation}
where $F_i\in\V_h^{\ast}$ is the linear functional
\begin{align*}
F_i(\vv_h):=\frac{1}{\tau}(\vy_h^i,\vv_h)_{H^1(\Omega;\mathbb{R}^3)},
\end{align*}
and $\delta E_h[\vy_h](\vv_h)$ is the first variation of $E_h[\vy_h]$ in the direction $\vv_h\in\V_h$. The latter turns out to be {\it nonlinear} in $\vy_h$ due to the nonlinear structure of \eqref{eq:discrete-energy}. The explicit expression of $\delta E_h[\vy_h](\vv_h)$ is tedious to compute and is omitted in this paper, but is given in the the first version of our companion arXiv preprint \cite{bouck2022computation} along with the second variation needed for a Newton method. In fact, we propose a Newton-type algorithm to solve \eqref{eq:gf-onestep} at each step $i$ of the gradient flow. Choosing $\vy_h^{i,0}:=\vy_h^{i}\in\V_h$ and assuming $\vy_h^{i,n}\in\V_h$ is known for $n\ge0$, we compute the increment $\delta\vy_h^{i,n}\in\V_h$ from
\begin{equation}\label{eq:Newton-step}
  \delta^2 L^{i}_h[\vy^{i,n}_h](\delta\vy_h^{i,n},\vv_h)=-\delta L^{i}_h[\vy^{i,n}_h](\vv_h)
  \quad\forall \, \vv_h\in\V_h,
\end{equation}
and update $\vy^{i,n+1}_h=\vy^{i,n}_h+\delta\vy_h^{i,n}$. We point out that
\begin{equation}\label{eq:second-variation}
\delta^2 L^{i}_h[\vy_h](\vv_h,\vw_h) = \frac{1}{\tau} \big( \vv_h,\vw_h \big)_{H^1(\Omega;\mathbb{R}^3)} +
\delta^2 E_h[\vy_h](\vv_h,\vw_h)
\quad\forall \, \vy_h,\vv_h,\vw_h\in\V_h
\end{equation}
and $\delta^2 E_h[\vy_h](\vv_h,\vw_h)$ is accessible by straightforward but lengthy calculations \cite{bouck2022computation}. The Newton sub-iteration \eqref{eq:Newton-step} is linear in $\delta\vy_h^{i,n}\in\V_h$, and is 
stopped once
\begin{equation*}
\big|\delta L^{i}[\vy_h^{i,M}](\delta\vy_h^{i,M})\big|^{1/2}\le \tol_1,
\end{equation*} 
for some integer $M>0$ and a pre-determined tolerance parameter $\tol_1$.

We choose the next iterate of the gradient flow to be the output of the Newton sub-iteration, i.e, $\vy^{i+1}_h:=\vy^{i,M}_h$. We stop the nonlinear gradient flow provided
\begin{equation*}
\frac{1}{\tau} \big|E_h[\vy_h^{N}]-E_h[\vy_h^{N-1}] \big|~\le~\tol_2
\end{equation*} 
is valid for some $N>0$ and fixed tolerance $\tol_2$, and declare $\vy_h^{N}$ to be the output.

Energy stability of the gradient flow is guaranteed if the minimization problem \eqref{eq:gf-energy} is solved exactly. This is an intrinsic property of \textit{implicit} gradient flows.
\begin{proposition}[energy decrease property] Given $\vy^{i}_h\in\V_h$ for $i\ge0$, suppose $\vy^{i+1}_h\in\V_h$ solves the minimization problem \eqref{eq:gf-energy}. Then $E_h[\vy_h^{i+1}]\le E_h[\vy_h^{i}]$ with strict inequality if $\vy^{i+1}_h\ne\vy^{i}_h$. Moreover, for any $N\ge1$, there holds
\begin{equation}\label{eq:energy-stability-sum}
E_h[\vy^{N}_h]+\frac{1}{2\tau}\sum_{i=0}^{N-1}\|\vy^{i+1}_h-\vy^i_h\|^2_{H^1(\Omega;\mathbb{R}^3)}\le E_h[\vy^{0}_h].
\end{equation}
\end{proposition}
\begin{proof}
Since $\vy_h=\vy_h^i$ is an admissible function in \eqref{eq:gf-energy}, we deduce
\begin{equation}\label{eq:energy-stability-onestep}
\frac{1}{2\tau}\|\vy^{i+1}_h-\vy^i_h\|^2_{H^1(\Omega;\mathbb{R}^3)}+E_h[\vy^{i+1}_h]=L^i_h[\vy^{i+1}_h]\le L^i_h[\vy^{i}_h]=E_h[\vy^{i}_h].
\end{equation}
This proves that the energy $E_h[\vy_h^i]$ is strictly decreasing provided $\vy^{i+1}_h\ne\vy^{i}_h$ and, upon summation from $i=0$ to $N-1$, also yields \eqref{eq:energy-stability-sum}. 
\end{proof}

We note that the energy decrease property is also valid for flow metrics other than $H^1(\Omega;\mathbb{R}^3)$. However, the choice of $H^1(\Omega;\mathbb{R}^3)$ allows for better control of the Newton sub-iterations solving \eqref{eq:gf-onestep}, which we now address. We base our comments below on our numerical experiments of Section \ref{sec:simulations} and the companion paper \cite{bouck2022computation}. 
\begin{itemize}  
\item \textbf{Initialization.}
When $\vy_h^{i}\in\V_h$ is given in the gradient flow outer iteration, it is natural to choose $\vy_h^{i,0}:=\vy_h^{i}$ as initial guess for the Newton's inner iterations that is designed to compute $\vy_h^{i+1}$. If $\vy_h^{i,\ast}$ is a local minimizer of \eqref{eq:gf-energy} and $E_h[\vy_h^0]\le\alpha$, then \eqref{eq:energy-stability-onestep} implies
\begin{equation*}
\frac{1}{2\tau}\|\vy_h^{i,\ast}-\vy_h^{i}\|^2_{H^1(\Omega;\mathbb{R}^3)}\le E_h[\vy^{i}_h]\le E_h[\vy^{0}_h]\le\alpha,
\end{equation*}
whence the $H^1$-distance between $\vy_h^{i,0}$ and the minimizer $\vy_h^{i,\ast}$ is proportional to $\tau^{1/2}$. This not only reveals the crucial role of $\tau$ but also of the $H^1$-metric for the discrete flow \eqref{eq:gf-energy}, which is the norm governing the stretching energy \eqref{eq:stretching-energy}.
\smallskip
\item \textbf{Well-posedness and convergence.}
In view of \eqref{eq:second-variation}, the quadratic structure of the flow metric term $\tau^{-1}(\cdot,\cdot)_{H^1(\Omega;\mathbb{R}^3)}$ may compensate for the lack of ellipticity of $\delta^2 E_h[\vy_h](\cdot,\cdot)$ due to the lack of convexity of $E_h$, provided $\tau$ is sufficiently small. Therefore, we expect well-posedness and superlinear convergence of the proposed Newton method, when $\tau$ is small.
\item \textbf{Moderate condition on $\tau$.} Our simulations of Section \ref{sec:simulations}, and those in \cite{bouck2022computation}, confirm solvability and convergence of the Newton sub-iterations \eqref{eq:Newton-step} with moderate values of $\tau$ relative to the meshsize $h$. Consequently, the restriction on $\tau$ is mild for current simulations and yet prevents the use of backtracking techniques.
\end{itemize}

\section{Convergence of discrete minimizers}\label{sec:numerical-analysis}

This section is dedicated to proving convergence of discrete minimizers under the following regularity assumption.

\begin{assumption}[regularity]\label{as:H2-immersibility}
The metric $g$ defined in \eqref{eq:target-metric} admits an $H^2$ isometric immersion:  there exists a $\vy\in H^2(\Omega;\mathbb{R}^3)$ such that $\nabla\vy^T\nabla\vy = g$ a.e.\ in $\Omega$ or equivalently $E[\vy]=0$.
\end{assumption}

Under regularity Assumption \ref{as:H2-immersibility}, the main result can be stated as follows.
\begin{theorem}[convergence of minimizers]\label{thm:compactness}
Let Assumption \ref{as:H2-immersibility} hold and 
let $\vy_h$ be a minimizer of $E_h$ with mean value $\overline{\vy}_h:=|\Omega|^{-1}\int_{\Omega}\vy_h \,d\vx$. Then there is a subsequence (not relabeled) of $\vy_h-\overline{\vy}_h$ that converges in $H^1(\Omega,\mathbb{R}^3)$ strongly to a function $\vy^*\in H^2(\Omega;\mathbb{R}^3)$ that satisfies $E[\vy^*]=0$, i.e. $\vy^*$ is an isometric immersion $\I[\vy^*]={\nabla\vy^{*}}^T\nabla\vy^{*} = g$.
\end{theorem}
We start with a roadmap to the proof of convergence of discrete minimizers, which is inspired by the seminal work \cite{friesecke2002theorem}. The first step is to build a recovery sequence $\vy_h$ for the isometric immersion $\vy\in  H^2(\Omega;\mathbb{R}^3)\cap W^{1,\infty}(\Omega;\mathbb{R}^3)$ in Assumption \ref{as:H2-immersibility}, that exhibits the desired energy scaling $E_h[\vy_h]\lesssim h^2$ (see Proposition \ref{prop:energy-scaling}.) For such a $\vy$ we know that $J[\vy] = \lambda\geq c_{s,s_0}> 0$ due to \eqref{eq:lambda-bounds}. The challenge is to show a similar lower bound for $J[\vy_h]=\det\nabla\vy_h^T\nabla\vy_h$. This is trickier than interpolating directly or after convolution because this procedure would not lend to a lower bound for $J[\vy_h]$ for $\vy$ merely in $W^{1,\infty}(\Omega;\mathbb{R}^3)$. Moreover, the regularity $\vy\in H^2(\Omega;\mathbb{R}^3)$ is borderline for $\Omega\subset\mathbb{R}^2$ and does not yield further pointwise regularity of $\nabla\vy$. Therefore, we instead resort to a Lusin approximation argument for Sobolev functions similar to that used in \cite{friesecke2002theorem}. To achieve the desired energy scaling of $E_h[\vy_h]\lesssim h^2$, we exploit both frame indifference and the neo-Hookean structure of the stretching energy in \eqref{eq:stretching-energy-density}.
\begin{equation*}
\int_\Omega W(\vx ,\nabla\vy_h)d\vx  = \int_\Omega \left(\big|\vL_{\vn_h}^{-1/2}[\nabla \vy_h,\; \vb_h]\vL_{\vm}^{1/2}\big|^2 - 3\right)\text{ }d\vx ,
\end{equation*}
where $\vn_h = \frac{\nabla \vy_h\vm}{|\nabla \vy_h\vm|}$ and $\vb_h = \frac{\partial_1\vy_h\times \partial_2\vy_h}{|\partial_1\vy_h\times \partial_2\vy_h|^2}$. We next recall that Remark \ref{rmk:str-SO3} (special rotations) implies that $\vR = \vL_{\vn}^{-1/2}[\nabla \vy,\; \vb]\vL_{\vm}^{1/2}\in SO(3)$, because $E[\vy]=0$ according to Proposition \ref{prop:target-metric} (target metric). This enables us to use the rotation $\vR$ to rewrite the integrand as
\[
\big|\vL_{\vn_h}^{-1/2}[\nabla \vy_h,\; \vb_h]\vL_{\vm}^{1/2} \big| = \big|\vR+\vA_h|,
\quad
\vA_h := \vL_{\vn_h}^{-1/2}[\nabla \vy_h,\; \vb_h]\vL_{\vm}^{1/2} - \vR,
\]
and invoke frame indifference. Multiplying by $\vR^T$ does not change the energy, i.e.
\begin{equation*}
\int_\Omega W(\vx ,\nabla\vy_h)d\vx  
= \int_\Omega \left(\big|\vR^T\vR+\vR^T\vA_h\big|^2 - 3\right)\text{ }d\vx  = \int_\Omega \left(\big|\Id_3+\vR^T\vA_h\big|^2 - 3\right)\text{ }d\vx .
\end{equation*}
Lemma \ref{lem:expansion} below and an $L^\infty$ bound on $\vR$ leads to a quadratic expansion around $\Id_3$:
\begin{equation*}
\int_\Omega W(\vx ,\nabla\vy_h)d\vx  = \int_\Omega\left( \big|\Id_3+\vR^T\vA_h\big|^2 - 3\right)\, d\vx \lesssim \Vert \vA_h\Vert^2_{L^2(\Omega; \mathbb{R}^{3\times3})}.
\end{equation*}
Finally, an error estimate on $\vy-\vy_h$ and properties of $\vy_h$ in Lemma \ref{lem:lagrange-approx} further imply that $\Vert \vA_h\Vert^2_{L^2(\Omega;\mathbb{R}^{3\times3})}\lesssim h^2$ and $|\vy_h|_{H^2_h(\Omega;\mathbb{R}^3)}^2\lesssim 1$, whence $E_h[\vy_h]\lesssim h^2$ according to \eqref{eq:discrete-energy}.

  Existence of a recovery sequence $\vy_h$ so that $E_h[\vy_h]\lesssim h^2$ implies that global discrete minimizers $\vy^*_h$ are uniformly bounded in the $H^2_h$-seminorm. The uniform bound means that a subsequence of $\vy_h^*$ converges strongly in $H^1(\Omega)$, which bypasses the convexity issues of $W$. The tools transfer a discrete $H^2$-bound to additional compactness have been developed for bending problems \cite{bonito2021dg}. We go over the relevant results in Lemmas \ref{lem:interp-pw-const} and \ref{lem:compactness-properties}.

We now connect our work with the existing literature. As in \cite{friesecke2002theorem}, energy scaling brings additional compactness, but the mechanism in this paper is $H^2$-regularity of isometric immersions rather than the geometric rigidity result in \cite{friesecke2002theorem}. We refer to \cite{plucinsky2018actuation} for a geometric rigidity result in the context of LCEs. Moreover, we learned from \cite{plucinsky2018actuation} that $\vR$ is a suitable rotation to exploit frame indifference and perform a quadratic expansion of $\big|\Id_3+\vR^T\vA_h\big|^2$ around the identity.

\subsection{Preliminaries}

This section covers preliminaries to lay the groundwork for the main results later. Subsection \ref{S:prelim-energy-scaling} contains preliminaries on how to approximate an $H^2$-isometric immersion of $g$. The key question is as follows:
\begin{equation}\label{key_question}
\begin{minipage}{.85\textwidth}
\emph{Given $\vy\in H^2(\Omega;\mathbb{R}^3)$ that satisfies $\nabla\vy^T\nabla\vy = g$, how does one construct $\vy_h\in \mathbb{V}_h$ such that $\Vert \vy-\vy_h\Vert_{H^1(\Omega; \mathbb{R}^3)}\lesssim h$ and $J[\vy_h]>0$ a.e. in $\Omega$?}
\end{minipage}
\end{equation}
  This kind of approximation requires some control in $W^{1,\infty}(\Omega)$. To achieve control over $J[\vy_h]$ in $L^\infty$, we regularize $\vy$ with a $\vy^\mu\in W^{2,\infty}(\Omega;\mathbb{R}^3)$ such that $J[\vy^\mu]\geq c$. We note that there are works on approximating maps by smooth maps with well-defined normals. We refer to \cite[Proposition 4.1]{conti2006derivation}, where the approximation is in the $L^\infty$-norm rather than $H^1$. In our context, however, we deal with functions that have higher regularity than \cite{conti2006derivation}. Hence, we are able to take advantage of Lusin truncation of Sobolev functions and ideas used in the construction of a recovery sequence in \cite{friesecke2002theorem}.

Subsection \ref{S:prelim-compactness} discusses the regularization of a piecewise constant matrix field by an $H^1$-matrix field. This regularization provides additional compactness and relies on a quasi-interpolant that has been used in previous works on DG methods for bending problems \cite{bonito2021dg}. Our presentation is brief but self-contained.

\subsubsection{Preliminaries for energy scaling}\label{S:prelim-energy-scaling}
We first establish a quadratic expansion of the neo-Hookean formula around the identity, thereby slightly improving on \cite[Proposition A.2]{plucinsky2018actuation}.
\begin{lemma}[scaling of neo-Hookean formula near identity]\label{lem:expansion}
If $\vA\in\mathbb{R}^{3\times3}$ satisfies $\det(\Id_3+\vA) = 1$, then
\begin{equation*}
\big|\Id_3 +\vA\big|^2-3\leq 3\big|\vA\big|^2.
\end{equation*}
\end{lemma}
\begin{proof}
Since $\det(\Id_3 +\vA)=1$, we may apply Proposition \ref{prop:nondegen-neo-Hookean} to bound
\begin{equation*}
|\Id_3 +\vA|^2-3\leq 3 \,\mathrm{dist} \, \big(\Id_3+\vA,SO(3)\big)^2 \leq 3\big|\Id_3 +\vA - \Id_3\big|^2 = 3\big|\vA\big|^2,
\end{equation*}
which is the desired result.
\end{proof}

We next introduce, without proof, a truncation argument for Sobolev functions from \cite[Proposition A.2]{friesecke2002theorem}; this is a suitable form of Lusin theorem. The result in \cite[Proposition A.2]{friesecke2002theorem} is stated with boundary conditions but it is still valid without them. We also point to a similar result in \cite[Theorem 3.11.6]{ziemer1989weakly} as well as a Lipschitz truncation of $W^{1,1}_0(\Omega)$ functions \cite[Theorem 13]{diening2013finite}.

\begin{lemma}[truncation of $H^2$-functions]\label{lem:truncation}
Let $\vy\in H^2(\Omega;\mathbb{R}^3)$. There exists $\vy^\mu\in W^{2,\infty}(\Omega;\mathbb{R}^3)$ such that 
\begin{equation}
\Vert \vy^\mu\Vert_{W^{2,\infty}(\Omega;\,\mathbb{R}^3)}\leq C\mu	\label{eq:W2-inf-bound},
\end{equation}
and for $S_\mu := \{\vx\in \Omega : \vy(\vx )\neq \vy^\mu(\vx )\text{ or }\nabla\vy(\vx )\neq \nabla\vy^\mu(\vx )\}$ we have the estimate
\begin{equation}\label{eq:smu-measure-bound}
|S_\mu|\leq C\frac{\omega(\mu)}{\mu^2} 	
\end{equation}
on the measure, $|S_\mu|$, of $S_\mu$, where
\begin{equation*}
 \omega(\mu) = \int_{\{|\vy|+|\nabla \vy|+|D^2\vy|\geq \frac{\mu}{2}\}} \left(|\vy|+|\nabla \vy|+|D^2\vy|\right)^2d\vx 
\end{equation*}
satisfies $\omega(\mu)\to0$ as $\mu\to\infty$.
\end{lemma}
Motivated by the proof of \cite[Theorem 6.1(ii)]{friesecke2002theorem}, we refine Lemma \ref{lem:truncation} for our purposes. In our case, the isometric immersion $\vy$ given by Assumption \ref{as:H2-immersibility} satisfies $\vy\in H^2(\Omega;\mathbb{R}^3)\cap W^{1,\infty}(\Omega,\mathbb{R}^3)$ and $J[\vy]\geq c_{s,s_0}>0$ by virtue of $\I[\vy] = g$.

\begin{lemma}[truncation of $H^2$-functions with Lipschitz control]\label{lem:approx-by-W2inf}
If $\vy\in H^2(\Omega;\mathbb{R}^3)$ $\cap W^{1,\infty}(\Omega;\mathbb{R}^3)$ and $J[\vy]\geq c>0$, then the function $\vy^\mu\in W^{2,\infty}(\Omega;\mathbb{R}^3)$ given by Lemma \ref{lem:truncation} satisfies the following bounds for $\mu$ sufficiently large:
\begin{align}
\Vert \vy^\mu\Vert_{W^{2,\infty}(\Omega;\mathbb{R}^3)}&\leq C\mu,	\label{eq:W2-inf-bound2}\\
\Vert \vy^\mu\Vert_{H^2(\Omega;\mathbb{R}^3)}& \leq C \Vert\vy\Vert_{H^2(\Omega;\mathbb{R}^3)},\label{eq:H2-bound-mu}\\
\Vert \vy^\mu\Vert_{W^{1,\infty}(\Omega;\mathbb{R}^3)}&\leq C\big(1+\Vert \vy\Vert_{W^{1,\infty}(\Omega;\mathbb{R}^3)}\big), \label{eq:W1-inf-bound}\\
J[\vy^\mu] &\geq \frac{c}{2}, \label{eq:det-est}\\
\Vert \vy^\mu-\vy \Vert_{H^1(\Omega;\mathbb{R}^3)}&\leq C\big(1+\|\vy\|_{W^{1,\infty}(\Omega;\mathbb{R}^3)}\big)\frac{\sqrt{\omega(\mu)}}{\mu}, \label{eq:H1-est}
\end{align}
where $C$ are generic constants independent of the truncation parameter $\mu$.
\end{lemma}
\begin{proof}
We first invoke Lemma \ref{lem:truncation} (truncation of $H^2$ functions). For all $\mu>0$, there exists a $\vy^{\mu}\in W^{2,\infty}(\Omega;\mathbb{R}^3)$ such that $\vy^\mu = \vy$ and $\nabla \vy^\mu = \nabla\vy$ on a set $\Omega\setminus S_\mu$, where $|S_\mu|\leq C\omega(\mu)/\mu^2$ and $\lim_{\mu\to\infty} \omega(\mu) = 0$. Additionally, $\Vert \vy^\mu\Vert_{W^{2,\infty}(\Omega;\mathbb{R}^3)}\leq C\mu$, which is \eqref{eq:W2-inf-bound2}. 

We shall now prove that $\vy^\mu$ satisfies the asserted properties starting with \eqref{eq:H2-bound-mu}. Using properties of $\vy^\mu$ on the good and bad sets yields
\begin{align*}
\Vert \vy^\mu\Vert_{H^2(\Omega;\mathbb{R}^3)}^2 &= \int_{S_\mu}|\vy^\mu|^2+|\nabla \vy^\mu|^2+|D^2\vy^\mu|^2\, d\vx+ \int_{\Omega\setminus S_\mu}|\vy^\mu|^2+|\nabla \vy^\mu|^2+|D^2\vy^\mu|^2\, d\vx\\
&\leq C|S_\mu| \mu^2+\int_{\Omega\setminus S_\mu} |\vy|^2+|\nabla \vy|^2+|D^2\vy|^2\, d\vx
\leq C|S_\mu| \mu^2 +\Vert \vy\Vert_{H^2(\Omega;\mathbb{R}^3)}^2.
\end{align*}
Since $C|S_\mu|\mu^2 \leq C\omega(\mu)\leq C\Vert \vy\Vert_{H^2(\Omega;\mathbb{R}^3)}^2$, \eqref{eq:H2-bound-mu} follows immediately.

In order to show \eqref{eq:W1-inf-bound} and \eqref{eq:det-est}, we first note that $J[\vy^\mu(\vx )]\geq c$ and $|\nabla \vy^\mu(\vx )|\leq C$ clearly hold for a.e. $\! \vx\in \Omega\setminus S_\mu$.
We now focus on $\vx\in S_\mu$. First, we proceed similarly to the proof of \cite[Theorem 6.1(ii)]{friesecke2002theorem} to show that there exists $\delta>0$ such that $B(\vx ,R)\cap \Omega\setminus S_\mu \neq\emptyset$ for $R:=\delta\sqrt{C\omega(\mu)}\mu^{-1}$ and all $\vx\in S_\mu$. Otherwise, $B(\vx ,R)\cap \Omega = B(\vx ,R)\cap S_\mu$ and, since $\Omega$ is Lipschitz, there exists $A>0$ such that
\begin{equation*}
AR^2\leq |B(\vx ,R)\cap \Omega| = |B(\vx ,R)\cap S_\mu|\leq |S_\mu|\leq \frac{C\omega(\mu)}{\mu^2}.
\end{equation*}
Setting $\delta = \sqrt{(2/A)}$ 
produces a contradiction. Returning to $\vx\in S_\mu$, we pick a $\vz\in  \Omega\setminus S_\mu$ such that $|\vx-\vz|\leq R$ and employ \eqref{eq:W2-inf-bound} to write
\begin{equation}\label{eq:est-grad-ymu}
|\nabla \vy^\mu(\vx ) - \nabla \vy^\mu(\vz )|\leq C\mu R = C\mu\delta\sqrt{\omega(\mu)}\mu^{-1} = C\delta\sqrt{\omega(\mu)}.
\end{equation}
Therefore,
\begin{align*}
  |\nabla \vy^\mu(\vx )| &\leq |\nabla \vy^\mu(\vx ) - \nabla \vy^\mu(\vz )| +|\nabla \vy^\mu(\vz )|
  \\ &\leq C(\delta\sqrt{\omega(\mu)}+\Vert\nabla \vy\Vert_{L^\infty(\Omega;\mathbb{R}^{3\times2})}) \leq C(1+\Vert\nabla \vy\Vert_{L^\infty(\Omega;\mathbb{R}^{3\times2})}),
\end{align*}
for $\mu$ sufficiently large; this shows \eqref{eq:W1-inf-bound}. Moreover, 
\begin{equation}\label{eq:J-est-intermediate-1}
J[\vy^\mu(\vx )] = J[\vy^\mu(\vz )] - \big(J[\vy^\mu(\vz )] - J[\vy^\mu(\vx )]\big)\geq c - \big|J[\vy^\mu(\vz )] - J[\vy^\mu(\vx )]\big|.
\end{equation}
Exploiting the Lipschitz continuity of $\vy\mapsto J[\vy]$ in $W^{1,\infty}$ within a ball of radius proportional to $(1+\Vert \vy\Vert_{W^{1,\infty}(\Omega;\mathbb{R}^3)})$, and combining the estimates \eqref{eq:W1-inf-bound} and \eqref{eq:est-grad-ymu} for $\vy^\mu$ yields
\begin{equation*}\label{eq:J-est-intermediate-2}
\big|J[\vy^\mu(\vz )] - J[\vy^\mu(\vx )]\big|\leq C \big(1+\Vert\nabla \vy\Vert_{L^\infty(\Omega;\mathbb{R}^{3\times2})} \big)\delta\sqrt{\omega(\mu)},
\end{equation*}
whence the right-hand side is smaller than $c/2$ provided $\mu$ is sufficiently large. Inserting this back into \eqref{eq:J-est-intermediate-1} gives $J[\vy^\mu(\vx )] \geq \frac{c}{2}$, which is \eqref{eq:det-est}.

It remains to prove \eqref{eq:H1-est}. We first write the error $\Vert \vy^\mu-\vy \Vert_{H^1(\Omega;\mathbb{R}^3)}^2$ as
\begin{equation*}
\Vert \vy^\mu-\vy \Vert_{H^1(\Omega;\mathbb{R}^3)}^2 = \int_{S_\mu} |\vy^\mu-\vy|^2+|\nabla\vy^\mu - \nabla\vy|^2d\vx ,
\end{equation*}
according to the definition of $S_\mu$ in Lemma \ref{lem:truncation} (truncation of $H^2$ functions). The $W^{1,\infty}$-bound \eqref{eq:W1-inf-bound} on $\vy^\mu$ in conjunction with the estimate \eqref{eq:smu-measure-bound} on the measure of $S_\mu$ produces the bound
\begin{equation*}
\Vert \vy^\mu-\vy \Vert_{H^1(\Omega;\mathbb{R}^3)}^2 \leq C \big(1+\|\vy\|_{W^{1,\infty}(\Omega;\mathbb{R}^3)}\big)^2 |S_\mu| \leq C \big(1+\|\vy\|_{W^{1,\infty}(\Omega;\mathbb{R}^3)}\big)^2\,\frac{\omega(\mu)}{\mu^2}.
\end{equation*}
Taking the square root of both sides yields the desired estimate.
\end{proof}

\begin{remark}
\rm
The argument in the proof of Lemma \ref{lem:approx-by-W2inf} is similar to that in the proof of \cite[Theorem 6.1(ii)]{friesecke2002theorem}, while the key difference is the object of interest. We want control over $\Vert\nabla \vy^\mu\Vert_{L^\infty(\Omega;\mathbb{R}^{3\times2})}$ and $J[\vy^{\mu}]$, while \cite{friesecke2002theorem} needs the gradient of the recovery sequence to be in an $L^\infty$-neighborhood of $SO(3)$.
\end{remark}

\begin{remark}
\rm
We stress that the significance of \eqref{eq:H1-est} is to provide a rate of convergence in $H^1$ relative to the blow up of the parameter $\mu$ that controls the $W^{2,\infty}$ norm, for which it is crucial that $\vy\in W^{1,\infty}(\Omega;\mathbb{R}^3)$. If $\vy\in H^2(\Omega;\mathbb{R}^3)$ but not in $W^{1,\infty}(\Omega;\mathbb{R}^3)$, then Sobolev embedding combined with \eqref{eq:H2-bound-mu} gives the reduced rate for all $2<p<\infty$
\[
\|\vy^\mu-\vy\|_{H^1(\Omega;\mathbb{R}^3)} \le \|\vy^\mu-\vy\|_{W^{1,p}(\Omega;\mathbb{R}^3)} \big|S_\mu\big|^{\frac{p-2}{2p}} 
\le C \|\vy\|_{H^2(\Omega;\mathbb{R}^3)} \mu^{-1+2/p}.
\]

\end{remark}

The next few results deal with numerical preliminaries that are important for energy scaling. The next result says that interpolating an $H^2$ function gives a discrete function that has a uniform discrete $H^2$-bound. We present the proof for completeness but the argument can be found in the proof of \cite[Proposition 5.3]{bonito2021dg}.

\begin{lemma}[Lagrange interpolation stability in $H^2$]\label{lem:lagrange-interp}
Let $\vy\in H^2(\Omega;\mathbb{R}^3)$. Then the Lagrange interpolant $I_h\vy\in\V_h$ satisfies $|I_h\vy|_{H^2_h(\Omega;\mathbb{R}^3)}\lesssim |\vy|_{H^2(\Omega;\mathbb{R}^3)}$.
\end{lemma}
\begin{proof} 
Consider an arbitrary edge $e\in\Eh$, and its neighboring elements $T_1,T_2\in\Th$ and set $\omega_e = T_1\cup T_2$. Since $\vy\in H^2(\Omega;\mathbb{R}^3)$, the jump of $\nabla\vy$ across $e$ is zero. Then, by a trace inequality, interpolation estimate and the fact that $I_h\vy$ is linear on each element, we obtain for any component $y$ of $\vy$
\begin{align*}
\Vert \jump{\nabla I_hy} \Vert_{L^2(e;\mathbb{R}^{2})} &= \Vert \jump{\nabla I_hy - \nabla y} \Vert_{L^2(e;\mathbb{R}^{2})}\\
&\leq 
h^{-1/2}\Vert \nabla I_hy - \nabla y\Vert_{L^2(\omega_e;\mathbb{R}^{2})} + h^{1/2}\Vert D^2 I_hy - D^2y\Vert_{L^2(\omega_e;\mathbb{R}^{2\times2})} \\
&\lesssim h^{1/2}\Vert D^2y\Vert_{L^2(\omega_e;\mathbb{R}^{2\times2})}.
\end{align*}
Dividing both sides by $h^{1/2}$, squaring and summing over edges gives the assertion in view of definition \eqref{eq:discrete-H2}.
\end{proof}

We next establish other approximation properties of the Lagrange interpolant.

\begin{lemma}[discrete approximation of $H^2$-maps]\label{lem:lagrange-approx}
 Let $\vy\in H^2(\Omega; \mathbb{R}^3)$ satisfy $\vy\in W^{1,\infty}(\Omega;\mathbb{R}^3)$ and $J[\vy]\geq c$ a.e. in $\Omega$. For all $h>0$ sufficiently small, there exists $\vy_h\in \mathbb{V}_h$ such that $\|\vy_h\|_{W^{1,\infty}(\Omega;\mathbb{R}^3)}\lesssim 1 + \Vert\vy\Vert_{W^{1,\infty}(\Omega;\mathbb{R}^3)}$ and the following estimates are valid
\begin{align}
J[\vy_h]&\geq \frac{c}{4}, \label{eq:J-yh}\\
\Vert \vy_h-\vy\Vert_{H^1(\Omega;\mathbb{R}^3)} &\lesssim h \big(1+\|\vy\|_{W^{1,\infty}(\Omega;\mathbb{R}^3)}
+\Vert\vy\Vert_{H^2(\Omega;\mathbb{R}^3)} \big), \label{eq:y-yh-H1}\\
|\vy_h|_{H^2_h(\Omega;\mathbb{R}^3)} &\lesssim 1+\Vert\vy\Vert_{H^2(\Omega;\mathbb{R}^3)}. \label{eq:yh-H2}
\end{align}
\end{lemma}
\begin{proof}
We first invoke Lemma \ref{lem:approx-by-W2inf} (truncation of $H^2$-functions with Lipschitz control) with $\mu_h= \delta h^{-1}$ to regularize $\vy$ with a $\vy^{\mu_h}$; the constant $\delta>0$ will be determined soon. We choose $\vy_h  = I_h\vy^{\mu_h}$ to be the Lagrange interpolant of $\vy^{\mu_h}$. Since $|\vy^{\mu_h}|_{W^{2,\infty}(\Omega)}\leq C\mu_h$, in light of \eqref{eq:W2-inf-bound2}, a standard error estimate for the Lagrange interpolant gives the $W^{1,\infty}$-error estimate
\begin{equation*}
\Vert \nabla \vy_h - \nabla \vy^{\mu_h}\Vert_{L^\infty(\Omega;\mathbb{R}^{3\times2})}\lesssim h |\vy^{\mu_h}|_{W^{2,\infty}(\Omega;\mathbb{R}^{3})}\lesssim h \mu_h = \delta.
\end{equation*}
This, together with \eqref{eq:W1-inf-bound}, implies uniform $W^{1,\infty}$-bounds for $\vy_h, \vy^{\mu_h}$, which in turn yield the following error estimate for $J[\vy_h]$ because of the Lipschitz continuity of $\vy\mapsto J[\vy]$ in $W^{1,\infty}$ within balls of radius proportional to $(1+\Vert\vy\Vert_{W^{1,\infty}(\Omega;\,\mathbb{R}^3)})$
\begin{equation*}
\Vert J[\vy_h]-J[\vy^{\mu_h}]\Vert_{L^\infty(\Omega)}\leq C\delta.
\end{equation*}
We choose $\delta$ sufficiently small, so that $C\delta < \frac{c}{4}$. Hence, for this choice of $\delta$, we have
\begin{equation*}
J[\vy_h] \geq J[\vy^{\mu_h}] - \Vert J[\vy_h]-J[\vy^{\mu_h}]\Vert_{L^\infty(\Omega)} \geq J[\vy^{\mu_h}] - C\delta\geq \frac{c}{2} - \frac{c}{4} = \frac{c}{4},
\end{equation*}
provided $h$ is sufficiently small, and correspondingly $\mu_h=\delta h^{-1}$ is sufficiently large for \eqref{eq:det-est} to be valid. This proves the first assertion \eqref{eq:J-yh}.

For the second assertion \eqref{eq:y-yh-H1}, we apply the triangle inequality
\begin{equation*}
  \Vert \vy-\vy_h\Vert_{H^1(\Omega;\mathbb{R}^3)} \leq \Vert \vy-\vy^{\mu_h}\Vert_{H^1(\Omega;\mathbb{R}^3)}
  +\Vert \vy^{\mu_h}-\vy_h\Vert_{H^1(\Omega;\mathbb{R}^3)},
\end{equation*}
and observe that \eqref{eq:H1-est} from Lemma \ref{lem:approx-by-W2inf} implies
\[
\Vert \vy-\vy^{\mu_h}\Vert_{H^1(\Omega;\mathbb{R}^3)} \leq C\mu_h^{-1} \big(1+\|\vy\|_{W^{1,\infty}(\Omega;\mathbb{R}^3)}\big).
\]
For the remaining term we utilize a standard error estimate for the Lagrange interpolant, in conjunction with \eqref{eq:H2-bound-mu} from Lemma \ref{lem:approx-by-W2inf}, to arrive at
\[
\Vert \vy^{\mu_h}-I_h \vy^{\mu_h}\Vert_{H^1(\Omega;\mathbb{R}^3)}\lesssim h |\vy^{\mu_h}|_{H^2(\Omega;\mathbb{R}^3)}
\lesssim h \Vert\vy\Vert_{H^2(\Omega;\mathbb{R}^3)}.
\]
Combining the last two bounds with $\mu_h^{-1}=h\delta^{-1}\lesssim h$ yields the desired estimate
\begin{align*}
\Vert \vy-\vy_h\Vert_{H^1(\Omega;\mathbb{R}^3)} 
\lesssim h \big(1+\|\vy\|_{W^{1,\infty}(\Omega;\mathbb{R}^3)}+\Vert\vy\Vert_{H^2(\Omega;\mathbb{R}^3)} \big)
\end{align*}
because $\delta$ has already been fixed.
Finally, the uniform $H^2_h$-bound \eqref{eq:yh-H2} follows from Lemma \ref{lem:lagrange-interp} (Lagrange interpolant stability in $H^2$) and the $H^2$-bound \eqref{eq:H2-bound-mu} on $\vy^\mu$ in Lemma \ref{lem:approx-by-W2inf}. This completes the proof.
\end{proof}

\subsubsection{Preliminaries for compactness}\label{S:prelim-compactness}
We show now how to extract $H^1$-compactness for sequences of continuous piecewise linear functions which are
not naturally in $H^2(\Omega)$. We proceed by discrete regularization via Cl\'{e}ment interpolation as in \cite{bonito2021dg}.

Suppose first that we have a function $v\in L^1(\Omega)$. Given a generic node $z \in \mathcal{N}_h$, with corresponding star (or patch) $\omega_z$, let $\Vh(\omega_z)$ be the space of continuous piecewise linear functions over $\omega_z$. We define the local $L^2$-projection over $\Vh(\omega_z)$ as follows:
\begin{equation}\label{eq:local-L2-proj}
v_z \in \Vh(\omega_z): \qquad \int_{\omega_z}(v_z - v)v_h = 0 \quad\forall v_h\in \Vh(\omega_z);
\end{equation}
note that $v_z=v$ if $v\in \Vh(\omega_z)$. We define the Cl\'{e}ment interpolant $\calI_h v\in\Vh$ to be
\begin{equation}\label{eq:quasi-interp}
\calI_h v:= \sum_{z\in\mathcal{N}_h} v_z(z) \phi_z,
\end{equation}
where $\{\phi_z\}_{z\in\mathcal{N}_h}$ denotes the nodal basis of $\mathbb{V}_h$ associated with $z\in\mathcal{N}_h$.

\begin{lemma}[regularization of piecewise constant functions]\label{lem:interp-pw-const} 
If $v:\Omega\to\mathbb{R}$ is a piecewise constant function over $\Th$, then its piecewise linear quasi-interpolant $\calI_hv\in C^0(\overline{\Omega})$ defined in \eqref{eq:local-L2-proj} and \eqref{eq:quasi-interp} satisfies the error estimates
\begin{equation}\label{eq:est-quasi-interp}
\Vert v- \mathcal{I}_hv\Vert_{L^2(\Omega)} + h\Vert \nabla\mathcal{I}_hv\Vert_{L^2(\Omega;\mathbb{R}^{2})} \lesssim h \sqrt{\sum_{e\in \mathcal{E}_h}\frac{1}{h}\int_e\jump{v}^2}.
\end{equation}
\end{lemma}
\begin{proof}
  This is a corollary of \cite[Lemma 2.1]{bonito2021dg}.
\end{proof}

This lemma is instrumental to derive compactness properties from sequences of functions with uniform $H^2_h$-bounds. This is what we establish next. The proof follows the proof of \cite[Proposition 5.1]{bonito2021dg}, but we sketch it for completeness.

\begin{lemma}[compactness properties]\label{lem:compactness-properties} 
Let $\vy_h\in\V_h$ satisfy the uniform bounds 
$\Vert \nabla \vy_h\Vert_{L^2(\Omega;\mathbb{R}^{3\times2})}\lesssim 1$ and $|\vy_h|_{H^2_h(\Omega;\mathbb{R}^3)}\lesssim 1$. Then there exists $\vy\in H^2(\Omega;\mathbb{R}^3)$ such that a subsequence (not relabeled) of $\vy_h-\overline{\vy}_h$ converges strongly
\begin{equation*} 
(\vy_h-\overline{\vy}_h)\to \vy
\end{equation*}
in $H^1(\Omega;\mathbb{R}^3)$ as $h\to0$, where $\overline{\vy}_h:=|\Omega|^{-1}\int_{\Omega}\vy_h$ is the mean value of $\vy_h$.
\end{lemma}
\begin{proof}

Since $\vy_h$ satisfies the uniform bound $\Vert \nabla \vy_h\Vert_{L^2(\Omega;\mathbb{R}^{3\times2})}\lesssim 1$, Poincar\'e inequality further implies the uniform bound $\|\vy_h-\overline{\vy}_h\|_{H^1(\Omega;\mathbb{R}^3)}\lesssim 1$. Therefore, there is $\vy \in H^1(\Omega;\mathbb{R}^3)$ and a subsequence (not relabeled) of $(\vy_h-\overline{\vy}_h)$ such that $(\vy_h-\overline{\vy}_h)\to \vy$ strongly in $L^2(\Omega;\mathbb{R}^3)$ and weakly in $H^1(\Omega;\mathbb{R}^3)$.

To extract additional regularity of $\vy$, we consider $\vw_h = \calI_h(\nabla \vy_h) \in [\Vh]^{3\times2}$ with $\calI_h$ defined in \eqref{eq:quasi-interp}. In view of the uniform bound $|\vy_h|_{H^2_h(\Omega;\mathbb{R}^{3})} \lesssim 1$, \eqref{eq:est-quasi-interp} of Lemma \ref{lem:interp-pw-const} implies that $\vw_h$ is uniformly bounded in $H^1(\Omega;\mathbb{R}^{3\times2})$ and $\vw_h - \nabla \vy_h\to0$ strongly in $L^2(\Omega;\mathbb{R}^{3\times2})$, whence $\vw_h\to \nabla \vy$ weakly in $L^2(\Omega;\mathbb{R}^{3\times2})$. The uniform $H^1$-bound of $\vw_h$ means that a subsequence (not relabeled) of $\vw_h\to \nabla \vy$ strongly in $L^2(\Omega;\mathbb{R}^{3\times2})$ and $\nabla \vy \in H^1(\Omega;\mathbb{R}^{3\times2})$. Consequently, a subsequence (not relabeled) of $\nabla \vy_h\to \nabla \vy$ strongly in $L^2(\Omega;\mathbb{R}^{3\times2})$ and completes the proof.
\end{proof}

\subsection{Energy scaling and compactness}

Our next result is a crucial discrete energy scaling estimate. It states that if there is an $H^2$-deformation $\vy$ that satisfies the target metric (i.e. an $H^2$-isometric immersion), then the discrete energy $E_h[\vy_h]$ associated with the discrete approximation $\vy_h$ of Lemma \ref{lem:lagrange-approx} (discrete approximation of $H^2$ maps) scales like $E_h[\vy_h] \lesssim h^2$. In the language of $\Gamma$-convergence, this is a recovery sequence result.

\begin{proposition}[recovery sequence]\label{prop:energy-scaling}
If $\vy\in H^2(\Omega;\mathbb{R}^3) \cap W^{1,\infty}(\Omega;\mathbb{R}^3)$ is the deformation of Assumption \ref{as:H2-immersibility} (regularity), then for any $h$ sufficiently small there exists $\vy_h\in \mathbb{V}_h$ such that
\begin{equation}
E_h[\vy_h]\lesssim h^2\big(1+\|\vy\|_{W^{1,\infty}(\Omega;\mathbb{R}^3)}+\Vert\vy\Vert_{H^2(\Omega;\mathbb{R}^3)} \big)^2.
\end{equation}
\end{proposition}
\begin{proof}
By Assumption \ref{as:H2-immersibility}, we know that $\vy\in H^2(\Omega;\mathbb{R}^3)$ satisfies $\nabla\vy^T\nabla\vy = g$, whence $E[\vy] = 0$ by Proposition \ref{prop:target-metric} (target metric), as well as $J[\vy] = \lambda\geq c_{s,s_0}>0$ by \eqref{eq:lambda-bounds}. By Lemma \ref{lem:lagrange-approx}, for $h$ sufficiently small, there exists $\vy_h\in \mathbb{V}_h$ such that $J[\vy_h(\vx )]\geq \frac{c_{s,s_0}}{4}$ and $|\vy_h|_{H^2_h(\Omega;\mathbb{R}^3)}\lesssim 1+ \Vert\vy\Vert_{H^2(\Omega;\mathbb{R}^3)}$. The latter implies that $R_h[\vy_h]=c_rh^2|\vy_h|_{H^2_h(\Omega;\mathbb{R}^3)}^2\lesssim h^2(1+ \Vert\vy\Vert_{H^2(\Omega;\mathbb{R}^3)})^2$ in \eqref{eq:discrete-energy}. It thus remains to show that $\int_\Omega W(\vx ,\nabla \vy_h)d\vx \lesssim h^2$, for which we resort to \eqref{eq:stretching-energy-density}
\begin{equation}\label{eq:neo-hookean-bis}
W(\vx ,\nabla \vy_h) = \big| \vL_{\vn_h}^{-1/2} [ \nabla \vy_h, \vb_h ] \vL_\vm^{1/2} \big|^2 - 3
\end{equation}   
where the kinematic constraint $\vn_h = \vn[\vy_h]$ and scaled normal is $\vb_h = \vb[\vy_h]$ are defined in \eqref{eq:shortenings}. We split the proof into three steps.

\medskip

{\it Step 1. Error estimate of scaled normal vectors.} We recall that these vectors are
\[
\vb = \frac{\partial_1\vy\times \partial_2\vy}{J[\vy]},
\qquad
\vb_h = \frac{\partial_1\vy_h\times \partial_2\vy_h}{J[\vy_h]},
\]
with $J[\vy] = |\partial_1\vy\times \partial_2\vy|^2$ and $J[\vy_h]=|\partial_1\vy_h\times \partial_2\vy_h|^2$. We claim that $|\vb - \vb_h|\lesssim |\nabla \vy - \nabla\vy_h|$ pointwise for which we write
\begin{align*}
\big|\vb - \vb_h\big| \leq \big|\partial_1\vy\times \partial_2\vy \big| \left|\frac{1}{J[\vy]} - \frac{1}{J[\vy_h]}\right|+ \frac{1}{J[\vy_h]} \big|\partial_1\vy\times \partial_2\vy - \partial_1\vy_h\times \partial_2\vy_h\big|.
\end{align*}
Since $J[\vy_h]\geq \frac{c_{s,s_0}}{4}$, according to \eqref{eq:J-yh}, the Lipschitz bound on $\vy$ yields
\begin{align*}
\big|\vb - \vb_h\big| \lesssim \left|\frac{1}{J[\vy]} - \frac{1}{J[\vy_h]}\right|+\big|\partial_1\vy\times \partial_2\vy - \partial_1\vy_h\times \partial_2\vy_h\big|.
\end{align*}
We now add and subtract $\partial_1\vy\times \partial_2\vy_h$, and apply the triangle inequality along with the bound $\|\vy_h\|_{W^{1,\infty}(\Omega;\mathbb{R}^3)}\lesssim 1 + \|\vy\|_{W^{1,\infty}(\Omega;\mathbb{R}^3)}$ from Lemma \ref{lem:lagrange-approx}, to further estimate
\begin{align*}
\big|\vb - \vb_h\big| \lesssim \left|\frac{1}{J[\vy]}- \frac{1}{J[\vy_h]}\right| + \big|\nabla\vy-\nabla\vy_h\big|.
\end{align*}
Since $x\mapsto \frac{1}{x}$ is Lipschitz on $[\frac{c_{s,s_0}}{4},\infty)$, we deduce
\begin{equation*}
\left|\frac{1}{J[\vy]}- \frac{1}{J[\vy_h]}\right|\lesssim \big|J[\vy]- J[\vy_h]\big|.
\end{equation*}
Likewise, on bounded subsets of $\mathbb{R}^{3\times2}$, the map $\vF\mapsto J(\vF)$ is Lipschitz. Hence, we again use the uniform $W^{1,\infty}$-bound of $\vy_h$ from Lemma \ref{lem:lagrange-approx} to obtain
\begin{equation*}
\left|J[\vy]- J[\vy_h]\right|\lesssim \big|\nabla\vy-\nabla\vy_h\big|.
\end{equation*}
Combining these bounds gives the desired pointwise error estimate for the scaled normals with hidden constant proportional to $(1+\Vert \vy\Vert_{W^{1,\infty}(\Omega;\mathbb{R}^3)})$
\begin{equation*}
|\vb - \vb_h|\lesssim |\nabla \vy - \nabla\vy_h|.
\end{equation*}

{\it Step 2. Estimate on the kinematic constraint.} Since $J[\vy_h]=\det \I[\vy_h] \ge \frac{c_{s,s_0}}{4}$, according to \eqref{eq:J-yh}, we deduce that $\I[\vy_h]$ is uniformly positive definite and
\[
\big| \nabla\vy_h \vm  \big|^2 = \vm^T \nabla\vy_h^T \nabla\vy_h \vm = \vm^T \I[\vy_h] \vm \ge c'
\]
for a constant $c'>0$ depending on $c_{s,s_0}$; a similar estimate is valid for $\big| \nabla\vy \vm  \big|$. Since $\vy_h \in W^{1,\infty}(\Omega;\mathbb{R}^3)$ is uniformly bounded, in view of Lemma \ref{lem:lagrange-approx}, and the map $\vx\mapsto \vx/|\vx|$ is Lipschitz on bounded subsets of $\{\vx\in \mathbb{R}^2: |\vx|\geq \sqrt{c'}\}$, we obtain the following pointwise bound with hidden constant proportional to $(1+\Vert \vy\Vert_{W^{1,\infty}(\Omega;\mathbb{R}^3)})$
\begin{equation*}
|\vn-\vn_h| = \left|\frac{\nabla \vy\vm}{|\nabla \vy\vm|} - \frac{\nabla \vy_h\vm}{|\nabla \vy_h\vm|} \right|\lesssim |\nabla \vy\vm - \nabla \vy_h\vm|\leq |\nabla \vy - \nabla \vy_h|.
\end{equation*}

{\it Step 3. Energy Scaling.} We now rewrite the neo-Hookean relation \eqref{eq:neo-hookean-bis} of $W(\vx ,\nabla\vy_h)$ as follows after adding and subtracting $\vR := \vL_{\vn}^{-1/2}[\nabla\vy,\; \vb]\vL_\vm^{1/2}\in SO(3)$:
\[
W(\vx ,\nabla\vy_h) = \big| \vR + \vA_h \big|^2 - 3, \quad
\vA_h := \vL_{\vn_h}^{-1/2}[\nabla\vy_h,\; \vb_h]\vL_\vm^{1/2}-\vL_{\vn}^{-1/2}[\nabla\vy,\; \vb]\vL_\vm^{1/2}.
\]
The fact that $\vR\in SO(3)$ is a consequence of Remark \ref{rmk:str-SO3} (special rotations) provided $\I[\vy] = g$ or equivalently $W(x,\nabla\vy) = 0$.
We exploit frame indifference to multiply by $\vR^T$ without changing the energy density
\begin{align*}
W(\vx ,\nabla\vy_h) = |\vR^T\vR +\vR^T \vA_h|^2-3 = |\Id_3 +\vR^T\vA_h|^2-3.
\end{align*}
Arguing as in the proof of Corollary \ref{cor:stretching-nondegeneracy}, we see that $\det \vL_{\vn_h} = \det \vL_\vm =\det [\nabla\vy_h,\; \vb_h] = 1$ and deduce $\det\left(\Id_3 +\vR^T\vA_h\right) = \det(\vR^T\vL_{\vn_h}^{-1/2}[\nabla\vy_h,\; \vb_h]\vL_\vm^{1/2})=1$. Applying Lemma \ref{lem:expansion} (scaling of neo-Hookean formula near identity), we obtain
\begin{align*}
\int_\Omega W(\vx ,\nabla\vy_h)d\vx  = \int_\Omega |\Id_3 +\vR^T\vA_h|^2-3\text{ }d\vx  \leq 3\int_\Omega |\vR^T\vA_h|^2 \leq 3\int_\Omega |\vA_h|^2.
\end{align*}
It thus suffices to show $\int_\Omega |\vA_h|^2d\vx \lesssim h^2$. Adding and subtracting $\vL_{\vn_h}^{-1/2} [\nabla\vy,\vb] \vL_\vm^{1/2}$, and using the triangle and Young's inequalities, yields
\begin{align*}
  |\vA_h|^2 &\lesssim \big|\vL_{\vn_h}^{-1/2}([\nabla\vy_h,\; \vb_h]-[\nabla\vy,\; \vb])\vL_\vm^{1/2}\big|^2
  + \big| (\vL_{\vn}^{-1/2}- \vL_{\vn_h}^{-1/2})[\nabla\vy,\; \vb]\vL_\vm^{1/2} \big|^2\\
  &\lesssim \big|[\nabla\vy_h,\; \vb_h]-[\nabla\vy,\; \vb]\big|^2 + \big|\vL_{\vn}^{-1/2}- \vL_{\vn_h}^{-1/2} \big|^2
  \lesssim \big|\nabla\vy- \nabla\vy_h\big|^2,
\end{align*}
where the last inequality follows from the preceding steps. In fact, Step 1 implies
\[
\big|[\nabla\vy_h,\; \vb_h]-[\nabla\vy,\; \vb]\big| \lesssim \big|\nabla\vy_h-\nabla\vy\big|,
\]
while Step 2, together with \eqref{eq:Ln-inv}, the assumptions on $s$ in \eqref{eq:s-lower-bound}, and $s\in L^\infty(\Omega)$ gives
\[
\big|\vL_{\vn}^{-1/2}- \vL_{\vn_h}^{-1/2}\big| \lesssim \big|\vn- \vn_h\big| \lesssim \big|\nabla\vy- \nabla\vy_h\big|,
\]
with hidden constant proportional to $(1+\Vert\vy\Vert_{W^{1,\infty}(\Omega;\mathbb{R}^3)})$.
Finally, applying \eqref{eq:y-yh-H1} of Lemma \ref{lem:lagrange-approx} (discrete approximation of $H^2$-maps) yields
\begin{equation}
 \int_\Omega |\vA_h|^2d\vx \lesssim \int_\Omega |\nabla\vy- \nabla\vy_h|^2d\vx  = \Vert\vy-\vy_h\Vert_{H^1(\Omega;\mathbb{R}^3)}^2 \lesssim h^2,
\end{equation}
with hidden constant proportional to $\big(1+\|\vy\|_{W^{1,\infty}(\Omega;\mathbb{R}^3)}+\Vert\vy\Vert_{H^2(\Omega;\mathbb{R}^3)} \big)^2$.
This is the desired estimate.
\end{proof}

\begin{remark}[regularity of $\vm$]\label{rmk:regularity-m}
\rm
It is worth realizing that the proof of Proposition \ref{prop:energy-scaling} only requires regularity on $\vy$, but not of $\vm$ beyond $L^\infty(\Omega;\mathbb{S}^1)$. We stress, however, that $\vy \in H^2(\Omega;\mathbb{R}^3) \cap W^{1,\infty}(\Omega;\mathbb{R}^3)$ implies $g=\nabla\vy^T\nabla\vy \in H^1(\Omega;\mathbb{R}^{2\times2})\cap L^{\infty}(\Omega;\mathbb{R}^{2\times2})$ with $g$ given in \eqref{eq:target-metric} in terms of $\vm$. This regularity is borderline and does not guarantee continuity of $g$ (or $\vm$) in $\Omega$.
\end{remark}

The next Proposition establishes compactness: if a discrete deformation $\vy_h$ satisfies an appropriate energy scaling, then a subsequence converges to a minimizer of $E$.
\begin{proposition}[compactness]\label{prop:selection-mechanism}
Let $\vy_h\in\V_h$ satisfy  $E_h[\vy_h]\leq Ch^2$ for a positive constant $C$, and let $\overline{\vy}_h:=|\Omega|^{-1}\int_{\Omega}\vy_h \,d\vx$ be its mean value. Then there is a subsequence (not relabeled) of $\vy_h-\overline{\vy}_h$ that converges in $H^1(\Omega;\mathbb{R}^3)$  strongly to a limit $\vy^*\in H^2(\Omega;\mathbb{R}^3)$ and $E[\vy^*]=0$.
\end{proposition}

\begin{proof} 
Proposition \ref{prop:coercivity} (coercivity) implies that $\Vert \nabla\vy_h\Vert_{L^2(\Omega;\mathbb{R}^{3\times2})}^2\lesssim 1$, whereas
\begin{equation*}
h^2|\vy_h|^2_{H^2_h(\Omega;\mathbb{R}^3)} \lesssim c_r  h^2|\vy_h|^2_{H^2_h(\Omega;\mathbb{R}^3)} + \int_\Omega W(\vx,\nabla \vy_h) = E_h[\vy_h]\lesssim h^2
\end{equation*}
yields $|\vy_h|^2_{H^2_h(\Omega;\mathbb{R}^3)}\lesssim 1$. Therefore, Lemma \ref{lem:compactness-properties} (compactness properties) guarantees the existence of $\vy^*\in H^2(\Omega;\mathbb{R}^3)$ such that a subsequence (not relabeled) $(\vy_h-\overline{\vy}_h)\to \vy^*$ converges strongly in $H^1(\Omega;\mathbb{R}^3)$. 
It remains to show that $E[\vy^*]=0$.

We can choose a further subsequence $\vy_h$ such that $\nabla\vy_h\to\nabla\vy^*$ a.e.\ in $\Omega$, whence
\[
J[\vy_h]\to J[\vy^*],
\quad
\nabla\vy_h\,\vm \to \nabla\vy^*\,\vm,
\quad
\textrm{a.e.\ in } \Omega.
\]
Our goal is to show that $\int_\Omega W(\vx,\nabla \vy_h)d\vx\to \int_\Omega W(\vx,\nabla \vy^*)d\vx$, for which we observe that 
\begin{equation*}
W(\vx,\nabla\vy_h) + 3 = \left|\vL_{\vn_h}^{-1/2}[\nabla\vy_h,\vb_h]\vL_\vm^{1/2}\right|^2 \ge \frac{C(s,s_0)}{J[\vy_h]},
\end{equation*}
which is a by-product of the proof of Proposition \ref{prop:coercivity} (coercivity) with $C(s,s_0)$ is the constant in \eqref{eq:str-coercivity}.
We first show that $J[\vy_h]$ does not vanish and the singular term $\frac{1}{J[\vy_h]}$ is well defined. If $B_{h,\eta} := \{ x\in \Omega : J[\vy_h] < \frac{1}{\eta}\}$, then we obtain
\begin{equation*}
  W(\vx,\nabla\vy_h) \geq \frac{C(s,s_0)}{J[\vy_h]}  - 3\geq \eta C(s,s_0) - 3
  \quad\forall \, x\in B_{h,\eta},
\end{equation*}
where $\eta>3C(s,s_0)^{-1}$ is to be determined.
This implies that
\begin{equation*}
|B_{h,\eta}|\leq \frac{1}{\eta C(s,s_0) - 3}\int_{B_{h,\eta}}W(\vx,\nabla\vy_h)\;d\vx\leq \frac{E_h[\vy_h]}{\eta C(s,s_0) - 3}\leq  \frac{Ch^2}{\eta C(s,s_0) - 3}.
\end{equation*}
Since $\nabla\vy_h$ is piecewise constant, $B_{h,\eta}$ is a collection of $N_\eta$ elements of the mesh $\Th$. By the shape regularity of $\Th$, there is $\gamma>0$ such that $|B_{h,\eta}| \geq N_\eta\gamma h^2$. Hence,
\begin{equation*}
N_\eta\gamma h^2 \leq \frac{Ch^2}{\eta C(s,s_0) - 3}.
\end{equation*}
Taking $\eta>0$ sufficiently large implies that $N_\eta=0$ and $J[\vy_h] \ge \frac{1}{\eta}$ a.e.\ in $\Omega$, whence we infer that $J[\vy^*] \ge \frac{1}{\eta}$ a.e.\ in $\Omega$. Note that since the matrix $[\vm,\vm_\perp]\in SO(2)$, we may rewrite $J[\vy]$ as
\begin{equation*}
J[\vy] = \det\left([\vm,\vm_\perp]^T\nabla\vy^T\nabla\vy[\vm,\vm_\perp]\right) =|\nabla \vy\,\vm|^2|\nabla \vy\,\vm_\perp|^2 - (\vm\cdot\I[\vy]\vm_\perp)^2.
\end{equation*}
As a result, we have $0< \frac{1}{\eta} \le J[\vy^*]\leq |\nabla\vy_h\vm|^2 |\nabla\vy_h\vm_\perp|^2$, and $|\nabla\vy_h\vm|^2>0$ a.e.\ in $\Omega$. 
Combined with continuity of $x\mapsto 1/x$ for positive $x$, we have $\frac{1}{J[\vy_h]}\to \frac{1}{J[\vy^*]}$ and $\vn_h = \frac{\nabla\vy_h\vm}{|\nabla\vy_h\vm|}\to \frac{\nabla\vy^*\vm}{|\nabla\vy^*\vm|} = \vn^*$ pointwise a.e. in $\Omega$. Thus, both 
$$\vL_{\vn_h}^{-1/2}[\nabla\vy_h,\vb_h]\vL_\vm^{1/2} \to \vL_{\vn^*}^{-1/2}[\nabla\vy^*,\vb^*]\vL_\vm^{1/2},$$
and 
$$W(\vx,\nabla \vy_h)\to W(\vx,\nabla \vy^*)$$ 
pointwise a.e.\ in $\Omega$.

Since $W(\vx,\nabla \vy_h)\geq 0$, by virtue of Corollary \ref{cor:stretching-nondegeneracy} (nondegeneray of stretching energy), we apply Fatou's Lemma to deduce the desired result
\begin{equation*}
E[\vy^*] = \int_\Omega W(\vx,\nabla \vy^*)d\vx \leq \liminf_{h\to0} \int_\Omega W(\vx,\nabla \vy_h)\, d\vx\leq\lim_{h\to0}E_h[\vy_h] = 0.
\end{equation*}
This concludes the proof.
\end{proof}

We are now ready to prove the convergence of discrete minimizers.
\begin{proof}[Proof of Theorem \ref{thm:compactness}]
The existence of a deformation $\vy$ satisfying Assumption \ref{as:H2-immersibility} (regularity) yields a quadratic energy scaling according to Proposition \ref{prop:energy-scaling} (recovery sequence). Since $\vy_h$ is a global minimizer of $E_h$, $E_h[\vy_h]\leq Ch^2$, and Proposition \ref{prop:selection-mechanism} (compactness) applies. Therefore, the limit $\vy^*\in H^2(\Omega;\mathbb{R}^3)$ satisfies $E[\vy^*] = 0$ and Proposition \ref{prop:target-metric} (target metric) implies that $\vy^*$ is an isometric immersion of $g$, i.e.\ $\I[\vy^*] = g$.
\end{proof}

\subsection{Piecewise $H^2$-deformations}
This section is dedicated to the analysis of piecewise $H^2$-deformations rather than globally $H^2$-deformations. The inspiration for this extension comes from \cite{bartels2022modeling} and \cite{bartels2022error}. For physical applications, the motivation comes from nonisometric origami \cite{plucinsky2016programming,plucinsky2018patterning,plucinsky2018actuation}.

Let $\Omega = \cup_{i=1}^n\Omega_i$ be a disjoint partition of $\Omega$, where each $\Omega_i$ is polygonal. We denote by $\Gamma$ the boundaries of all $\Omega_i$'s, which is the {\it set of creases or folding set}. We then define the space of piecewise $H^2$ functions to be
\begin{equation}\label{eq:pw-H2}
\mathbb{V}_\Gamma = \{ \vy\in W^{1,\infty}(\Omega;\mathbb{R}^3) : \vy|_{\Omega_i}\in H^2(\Omega_i;\mathbb{R}^3)\text{ for all } i=1,\ldots,n\}.
\end{equation}
We shall approximate minimizers $\vy^*\in\mathbb{V}_\Gamma$ of \eqref{eq:stretching-energy} with folding across $\Gamma$. To this end, we make the geometric assumption 
\begin{equation}\label{eq:fitted-mesh}
\Gamma\subset \bigcup_{e\in \mathcal{E}_h}e,
\end{equation}
 i.e. the mesh is fitted to $\Gamma$. We denote by $\mathcal{E}^i_h$ the \emph{interior} skeleton to each $\Omega_i$ (so that edges on $\Gamma$ are excluded) and define the new discrete energy with folds as
\begin{equation}\label{eq:discrete-energy-folds}
  E_{h,\Gamma}[\vy_h] :=\int_\Omega W_h(\vx ,\nabla\vy_h)d\vx  + R_{h,\Gamma}[\vy_h],
\end{equation}
where the regularization term is given by $R_{h,\Gamma}[\vy_h]:=c_rh^2 |\vy_h|^2_{H^2_h(\Omega\setminus\Gamma;\mathbb{R}^3)}$
and
\begin{equation}\label{eq:discrete-H2-Gamma}
|\vy_h|^2_{H^2_h(\Omega\setminus\Gamma;\,\mathbb{R}^3)} := \sum_{i=1}^n\sum_{e\in\mathcal{E}^{i}_h}\frac{1}{h}\int_e \big|\jump{\nabla \vy_h}\big|^2.
\end{equation}
We point out that \eqref{eq:discrete-H2-Gamma} does not include jumps across $\Gamma$, which in turn allows for folds across $\Gamma$ without penalty on the energy. This modeling feature is responsible for the formation of nonisometric origami within this setting.

We next adjust the regularity Assumption \ref{as:H2-immersibility} to the new framework.
\begin{assumption}[regularity with creases]\label{as:pwH2-immersibility}
There exists a $\vy\in W^{1,\infty}(\Omega;\mathbb{R}^3)$ such that $\I[\vy] = g$ a.e. in $\Omega$ and $\vy|_{\Omega_i}\in H^2(\Omega_i;\mathbb{R}^3)\cap C^1(\overline{\Omega}_i;\mathbb{R}^3)$ for all $i=1,\ldots,n$.
\end{assumption}

We relax the $H^2$-regularity but observe that $\vy|_{\Omega_i}\in C^1(\overline{\Omega}_i;\mathbb{R}^3)$ implies that $g|_{\Omega_i} \in C(\overline{\Omega}_i;\mathbb{R}^{2\times2})$ is slightly stronger than the mere $L^\infty\cap H^1$-regularity of $g$ as discussed in Remark \ref{rmk:regularity-m} (regularity of $\vm$). We point out that Assumption \ref{as:pwH2-immersibility} might not be always satisfied. It is possible that such a piecewise $H^2$-isometric immersion does not exist if one of $\Omega_i$ has reentrant corners.

We now state the new recovery sequence result.
\begin{proposition}[recovery sequence]\label{prop:recovery-pw-H2}
If $\vy\in\V_\Gamma$ is the deformation of Assumption \ref{as:pwH2-immersibility}, then for $h$ sufficiently small the Lagrange interpolant $\vy_h = I_h\vy \in \mathbb{V}_h$ satisfies
\begin{equation*}
E_{h,\Gamma}[\vy_h]\lesssim h^2 \left(1+\sum_{i=1}^n\left(\Vert\vy\Vert_{C^1(\overline{\Omega}_i;\mathbb{R}^3)} +\Vert\vy\Vert_{H^2(\Omega_i;\mathbb{R}^3)}\right)^2\right).
\end{equation*}
\end{proposition}

\begin{proof}
In view of \eqref{eq:fitted-mesh} and $\vy|_{\Omega_i} \in H^2(\Omega_i;\mathbb{R}^3)$ from Assumption \ref{as:pwH2-immersibility}, Lemma \ref{lem:lagrange-interp} (Lagrange interpolation stability in $H^2$) applied to each $\Omega_i$ gives  $|\vy_h|_{H^2_h(\Omega_i;\mathbb{R}^3)}\lesssim |\vy|_{H^2(\Omega_i;\mathbb{R}^3)}$. Moreover, we also have the standard error estimate $\Vert \vy - \vy_h\Vert_{H^1(\Omega_i;\mathbb{R}^3)}\lesssim h |\vy|_{H^2(\Omega_i;\mathbb{R}^3)}$.

To derive the energy scaling, we first show that $J[\vy_h]\geq c_{s,s_0}/2$ a.e. for sufficiently small $h$ provided $J[\vy]=\det g=\lambda \ge c_{s,s_0} > 0$. Since $\vy\in C^1(\overline{\Omega}_i;\mathbb{R}^3)$, the function $\nabla\vy$ is uniformly continuous in $\Omega_i$ with modulus of continuity $\sigma_i(t)$ (i.e. $\sigma_i(t)\to0$ as $t\to0$). Therefore, $\Vert\nabla\vy-\nabla\vy_h\Vert_{L^\infty(\Omega;\mathbb{R}^{3\times2})} \lesssim \sigma_i(h)$ and, for $h$ sufficiently small, we obtain
\begin{equation*}
J[\vy_h]\geq J[\vy] - \big|J[\vy_h] - J[\vy] \big|\geq c_{s,s_0} - C\sigma_i(h) \ge  \frac{c_{s,s_0}}{2}\quad \text{ in } \Omega_i
\end{equation*}
because $J$ is Lipschitz continuous in $W^{1,\infty}(\Omega;\mathbb{R}^3)$ on bounded balls.
Applying the arguments in Proposition \ref{prop:energy-scaling} (recovery sequence), we deduce
\begin{equation*}
\int_{\Omega_i} W(\vx ,\nabla\vy_h)d\vx  + c_r h^2|\vy_h|_{H^2_h(\Omega_i;\mathbb{R}^3)}^2 \lesssim h^2\left(1+\Vert\vy\Vert_{C^1(\overline{\Omega}_i;\mathbb{R}^3)} +\Vert\vy\Vert_{H^2(\Omega_i;\mathbb{R}^3)}\right)^2
\end{equation*}
on each $\Omega_i$. Summing over $\Omega_i$ yields the desired result.
\end{proof}

The compactness result in the previous section carries over to the case with jumps, but with a small modification. The analog to Theorem \ref{thm:compactness} reads as follows. 

\begin{theorem}[convergence of minimizers with creases]
Let Assumption \ref{as:pwH2-immersibility} hold and let $\vy_h$ be a global minimizer of $E_{h,\Gamma}$ with $\overline{\vy}_h=|\Omega|^{-1} \int_\Omega \vy_h$. Then, as $h\to0$, $\vy_h-\overline{\vy}_h$ has a strongly convergent subsequence (not relabeled) $\vy_h - \overline{\vy}_h\to \vy^*$ in $H^1(\Omega;\mathbb{R}^3)$ to a function $\vy^*\in \mathbb{V}_\Gamma$ that satisfies $E[\vy^*]=0$ and $\I[\vy^*] = g$ a.e.\ in $\Omega$.
\end{theorem}
\begin{proof}
We first apply Proposition \ref{prop:recovery-pw-H2} (recovery sequence) to deduce that $E_{h,\Gamma}[\vy_h]\leq E_{h,\Gamma}[I_h\vy] \lesssim h^2$ because $\vy_h$ is a global minimizer of $E_{h,\Gamma}$. Moreover, since $E[\vy_h] \le E_{\Gamma,h}[\vy_h] \lesssim h^2$ by definition \eqref{eq:discrete-energy-folds}, Proposition \ref{prop:coercivity} (coercivity) implies the uniform bound $\|\nabla\vy_h\|_{L^2(\Omega;\mathbb{R}^{3\times2})}\lesssim 1$ and, hence, the weak convergence of a subsequence (not relabeled) of $\vy_h - \overline{\vy}_h$ to a function $\vy^* \in H^1(\Omega;\mathbb{R}^3)$. We need to prove further regularity of $\vy^*$.

Proceeding now as in Lemma \ref{lem:compactness-properties} (compactness properties) and Proposition \ref{prop:selection-mechanism} (compactness) over each subdomain $\Omega_i$, we can show that up to a subsequence $\nabla\vy_h|_{\Omega_i}\to\nabla\vy^*|_{\Omega_i}$ converges strongly in $L^2(\Omega_i;\mathbb{R}^{3\times2})$ and that $\nabla\vy^*|_{\Omega_i}\in H^1(\Omega_i;\mathbb{R}^{3\times2})$  and $\I[\vy^*|_{\Omega_i}] = g$ a.e.\ in $\Omega_i$ for each $i=1,\ldots,n$. In view of Proposition \ref{prop:target-metric} (target metric) we also obtain that $W(\vx ,\nabla\vy^*|_{\Omega_i})=0$ for each $i=1,\ldots,n$, whence $E[\vy^*]=0$.

It remains to show that $\vy^*$ is globally Lipschitz, i.e.\ $\vy^* \in W^{1,\infty}(\Omega;\mathbb{R}^3)$. We note that $\vy^*|_{\Omega_i} \in W^{1,\infty}(\Omega_i;\mathbb{R}^3)$ for each $i=1,\ldots,n$ because $\I[\vy^*|_{\Omega_i}]=g\in L^\infty(\Omega_i;\mathbb{R}^{2\times2})$, which in turn implies that the trace of $\vy^*|_{\Omega_i}$ on $\partial\Omega_i$ is continuous. Since $\vy^* \in H^1(\Omega;\mathbb{R}^3)$, we infer that the jumps $[\vy^*]|_\Gamma=0$ must vanish, thereby showing that $\vy^* \in C^0(\overline{\Omega};\mathbb{R}^3)$ is uniformly continuous in $\Omega$. This, in addition to being piecewise Lipschitz, proves that $\vy^*$ is globally Lipschitz, whence   $\vy^*\in\V_\Gamma$ as asserted.
\end{proof}

\section{Numerical simulations}\label{sec:simulations}
We implement the proposed method within the multi-physics
finite element software Netgen/NGSolve \cite{schoberl2017netgen},
and the visualization relies on ParaView~\cite{Ahrens2005}.
In this section, we present several tests to illustrate properties of the LCN model \eqref{eq:stretching-energy}-\eqref{eq:stretching-energy-density}, as well as effectiveness and efficiency of our algorithm. A derivation of \eqref{eq:stretching-energy}-\eqref{eq:stretching-energy-density} and more extensive computational investigation is contained in \cite{bouck2022computation}.  

\subsection{Non-isometric origami: pyramids}\label{sec:pyramids}
In this subsection, we study an example of non-isometric origami, whose structure is made of folding thin sheets and complies with Assumption \ref{as:pwH2-immersibility} (regularity with creases). We refer to \cite{plucinsky2016programming, plucinsky2018actuation,plucinsky2018patterning} for a more detailed introduction of non-isometric origami. 

First, we divide the domain $\Omega$ into several subdomains $\Omega_i$ by ``folding lines'' or ``creases'' $\Gamma$, and consider meshes fitted to the folding lines. We take the regularization parameter $c_r=0$ along $\Gamma$ and $c_r=100$ in the rest of $\Omega$. In fact, vanishing regularization models a weakened  (or damaged) material on creases \cite{bartels2022modeling}, and mathematically this allows for the formation of kinks.

We consider piecewise constant blueprinted director fields $\vm$ and set-up creases $\Gamma$ and subdomains $\Omega_i$ as depicted in Fig. \ref{fig:pyramid_dir}. In this experiment, we take 
$\Omega=[0,1]^2$,
\[
h=1/64, \quad s=0.1, \quad s_0=1, \quad \tol_1=10^{-10}, \quad \tol_2=10^{-6}.
\]
The ensuing configurations are all compatible in the sense of \cite[formula (6.3)]{plucinsky2018patterning}: the magnitude of the tangential components of $\vm$ and $\lambda$ are continuous across $\Gamma$. We also discuss incompatible origami in \cite{bouck2022computation}.

\textit{Case 1.} We first consider the set-up on the left of Fig. \ref{fig:pyramid_dir}, $\tau=1$, and use initialization $\vy_h^0=I_h\vy^0$ with 
\begin{equation}\label{eq:initialization_pyramid_1}
\vy^0(x_1,x_2)~=~\bigg(x_1,x_2,0.8x_1(1-x_1)x_2(1-x_2)\bigg).
\end{equation} 

\smallskip
\textit{Case 2.}
We then consider the set-up on the right of Fig. \ref{fig:pyramid_dir}, $\tau=0.4$, and use the same initialization as \eqref{eq:initialization_pyramid_1}, 

\smallskip
\textit{Case 3.}
We also apply another initialization 
\begin{equation}\label{eq:initialization_pyramid_2}
\vy^0(x_1,x_2)~=~\bigg(x_1,x_2, 0.2\cos\big(7\pi(x_1-0.5)\big)x_2(x_2-1)\bigg)
\end{equation} 
to the set-up on the right of Fig. \ref{fig:pyramid_dir} and take $\tau=0.5$.
 
The computed solutions for all three cases are shown in Fig. \ref{fig:pyramid_final}. We get pyramid-like final configuration for \textit{Case 1}, which is consistent with the prediction in \cite{modes2011blueprinting}. For \textit{Cases 2} and \textit{3}, we obtain different equilibria starting from different initial states, but the difference in final energies is about $10^{-6}$. They are indeed global minimizers, because computed metric deviations 
\begin{equation}\label{eq:metric-deviation}
e_h[\vy^{\infty}_h]~:=~\| \I[\vy^{\infty}_h]~-~g\|_{L^1(\Omega)}
\end{equation}
are $1.6\times10^{-3}, 2.5\times10^{-3}, 2.4\times10^{-3}$ for \textit{Cases 1,2,3} respectively. 
Therefore, this gives an example where global minimizers to \eqref{eq:minimization-pb-discrete} are non-unique, and computed equilibrium shapes depend on initializations.      
This verifies the heuristic discussion in Example \ref{ex:pyramid-sawtooth}, confirms the lack of convexity of this model, and illustrates capability and accuracy of our numerical method for computing origami structures.    

\smallskip

\textit{Case 4.}
To confirm that the pyramid-like origami structure is \emph{not} a mesh effect, we generate a mesh with $h=1/64$ unfitted to the two diagonals $\Gamma$ of the square.
We consider the same set-up as in \textit{Case 1} except that the regularization parameter $c_r(\vx )=0$ if $\vx\in \Gamma_{0.02}$ and $c_r(\vx )=100$ otherwise, where $\Gamma_d:=\{\vx\in \Omega:\textrm{dist}(\vx ,\Gamma)<d\}$ is a strip surrounding the crease $\Gamma$.

The computed solution for the \textit{Case 4} is also displayed in Fig. \ref{fig:pyramid_final}. We still get the pyramid-like configuration, but with tiny wrinkling appearing in the strips $\Gamma_{0.02}$, due to the lack of regularization in this region. We present a thorough discussion of the computational effect of regularization in \cite{bouck2022computation}.
 
\begin{figure}[htbp]
\begin{center}
\includegraphics[width=.4\textwidth]{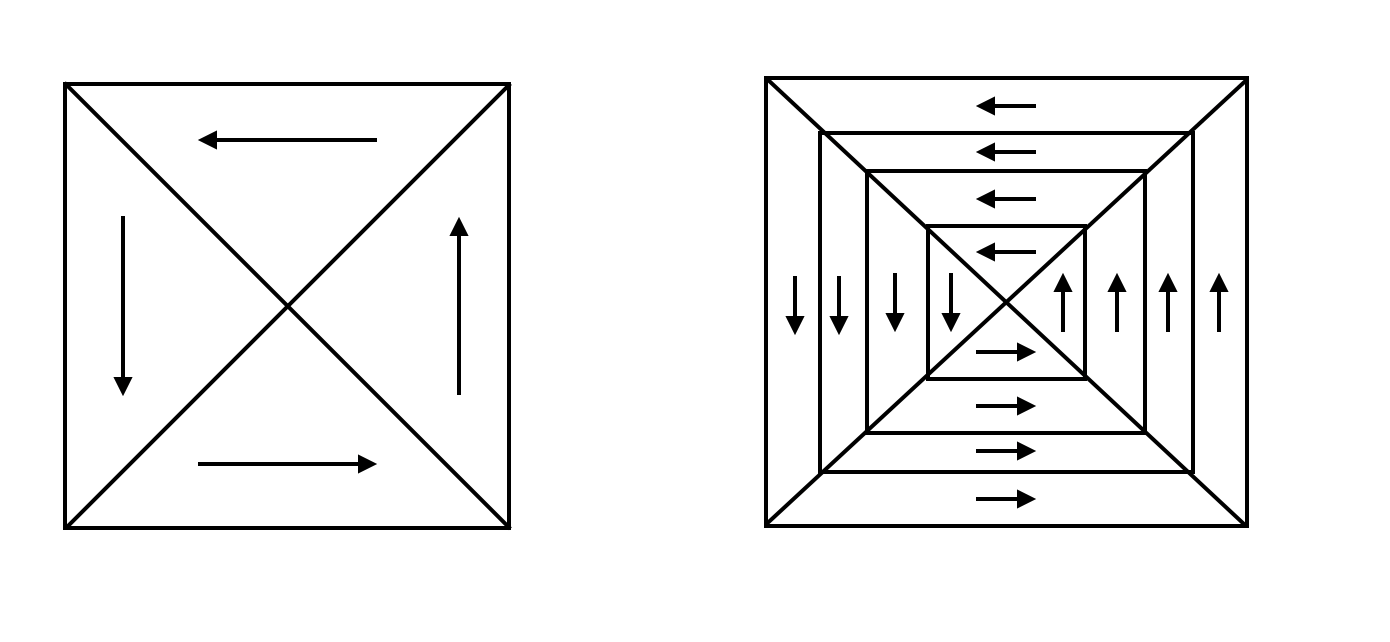}
\end{center}
\caption{This is the set-up for experiments in Subsection \ref{sec:pyramids}. Solid lines inside the square represent the locations of the creases, and arrows shows the piecewise constant director field $\vm$ in each subdomain. In this case, $\vm=(0,-1),(-1,0),(0,1),(1,0)$ in different subdomains.}
\label{fig:pyramid_dir}
\end{figure}

\begin{figure}[htbp]
\begin{center}
\includegraphics[width=.8\textwidth]{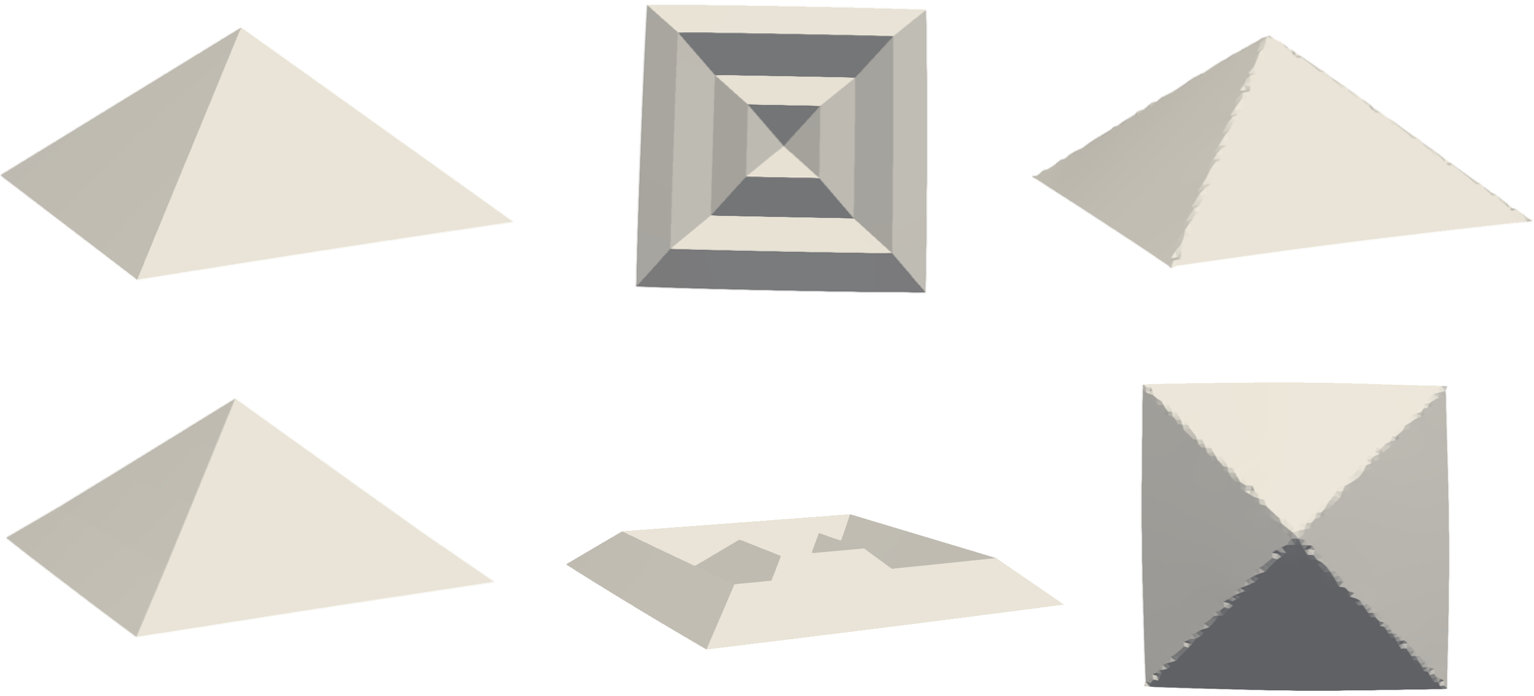}
\end{center}
\caption{Non-isometric origami: First column, pyramid-like final configurations for \textit{Case 1} and \textit{Case 2}. Second column, different views of final configuration for \textit{Case 3} exhibiting multiple folds. Third column, different views of final configuration for \textit{Case 4}, thereby confirming that the pyramid-like configuration is not a mesh effect.}
\label{fig:pyramid_final}
\end{figure}

\subsection{Liquid crystal defects}\label{sec:lc-defects}
In this section, we simulate a configuration arising from a liquid crystal defect, which is inspired by experimental results in \cite{mcconney2013topography, white2015programmable} and numerical simulations in \cite{chung2017finite}. We take $\vm$ in polar coordinates $(r,\theta)$ to be
\begin{equation}\label{eq:m_rot_sym_2}
\vm(r,\theta) = \big(\cos(1.5 \, \theta),\sin(1.5 \, \theta)\big).
\end{equation}
This rotationally symmetric blueprinted director field is discontinuous at the origin where it exhibits a stable defect of degree $3/2$. We consider $h = 1/128$ and $\tau$ to be the largest admissible $\tau$ so that the Newton sub-iteration is well-posed, an issue explored later in Section \ref{sec:method-quant}. The other parameters for the model are taken to be
\begin{equation}\label{eq:params_rot_sym_2}
s=0.1, \quad s_0=1, \quad c_r=1, \quad \tol_1=10^{-10}, \quad \tol_2=10^{-9}.
\end{equation}
The computed solution is displayed in Fig. \ref{fig:1p5defect}. Moreover, Fig.\ \ref{fig:conv-plot} (with $c_r=0$) reveals that the energy $E_h[\vy_h^\infty]$ decays subquadratically in $h$, which indicates that the limiting deformation $\vy^*$ is not in $H^2(\Omega;\mathbb{R}^3)$ whence it does not satisfy Assumption \ref{as:H2-immersibility} (regularity). Other configurations arising from liquid crystal defects have been computed in \cite{chung2017finite}, though to the best of our knowledge our simulation seems to be the first one of a defect of degree $3/2$. In \cite{bouck2022computation}, we present several configurations beyond theory, including higher order defects, which are computational accessible by our algorithm.
\begin{figure}[htbp]\begin{center}
\includegraphics[width=.3\textwidth]{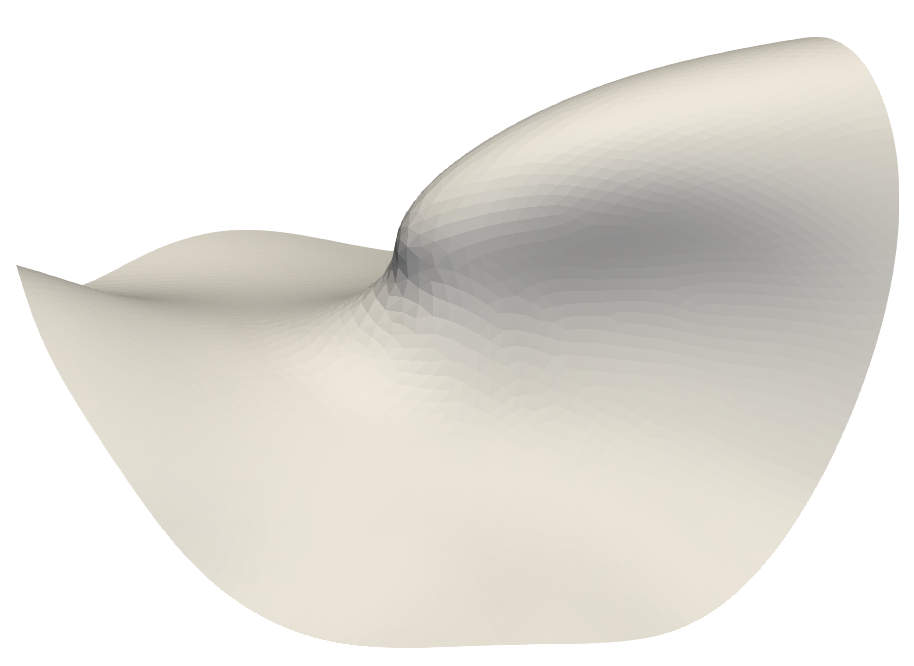}
\includegraphics[width=.28\textwidth]{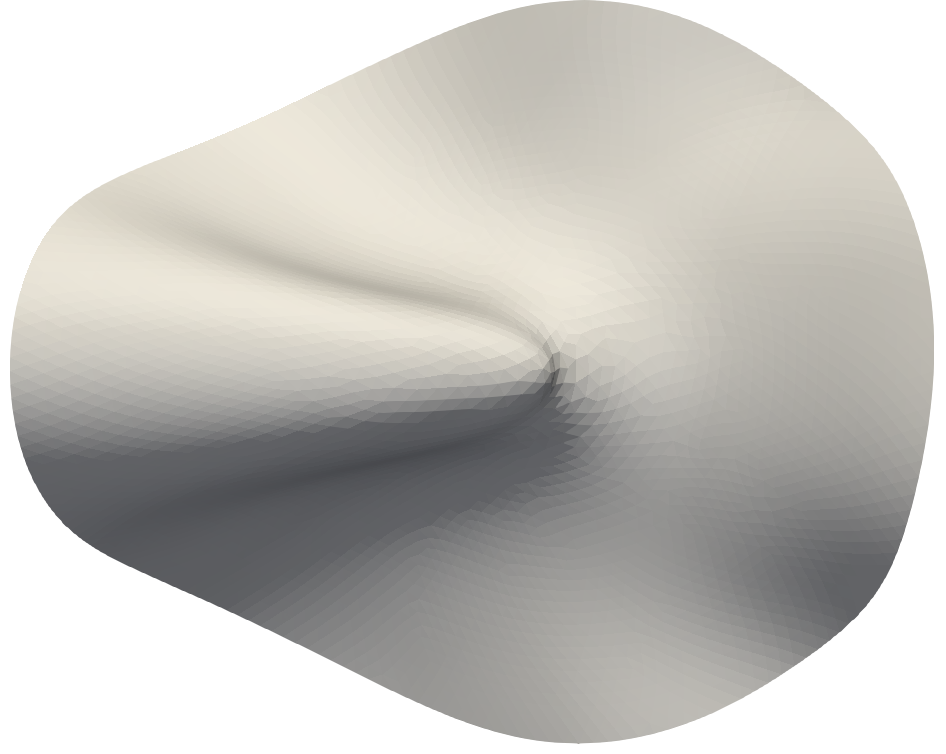}
\end{center}
\caption{Blueprinted director field $\vm$ with a stable defect of degree $3/2$: Two views of the computed deformation $\vy_h^{\infty}$. We observe a ``bird's beak'' structure around the defect location, which matches the experimental picture shown in \cite{mcconney2013topography, white2015programmable}.}
\label{fig:1p5defect}
\end{figure}

\subsection{Quantitative properties}\label{sec:method-quant}
In this subsection, we investigate computationally some quantitative properties of the proposed method, and in particular the role of meshsize $h$ and pseudo time step $\tau$. Our goals are as follows.
\begin{itemize}
\item \textbf{Convergence of metric deviation}. 
We measure the metric deviation $e_h[\vy_h^{\infty}]$ defined in \eqref{eq:metric-deviation} as an error between computed solutions $\vy_h^{\infty}$ and global minimizers, and recall that $g$ is given by \eqref{eq:target-metric}. We expect convergence of $e_h[\vy^{\infty}_h]$ as $h\to0$.
\item \textbf{Convergence of energy}. We can see that the exact minimum energy is $0$ from discussions of Section \ref{sec:properties-str-energy}. Therefore, we also expect convergence of the energy error $|E_h[\vy^{\infty}_h]|:=|E_h[\vy^{\infty}_h]-0|$ as $h\to0$.
\item \textbf{Role of pseudo time-step $\tau$}. We expect that the well-posedness and convergence of Newton method \eqref{eq:Newton-step} depend on $\tau$. We thus disclose the influence of $\tau$ on the final energy $E_h[\vy^{\infty}_h]$, metric deviation $e_h[\vy^{\infty}_h]$ and the number $N$ of gradient flow iterations. 
\end{itemize}

We consider three experiments to explore these issues computationally.

\medskip
\emph{Experiment 1: smooth $\vm$}. Let $\Omega$ be the unit square $\Omega=[-0.5,0.5]^2$ and 
\begin{equation}\label{eq:cont_dir_1}
{\vm}=(x_1+1,x_2+1)/\sqrt{(x_1+1)^2+(x_2+1)^2}. 
\end{equation}
We take parameters
\[
s=0.1, \quad s_0=1, \quad c_r=0, \quad \tol_1=10^{-10}, \quad \tol_2=10^{-9},
\]
and the initialization $\vy_h^0=I_h\vy^0$ with
\begin{equation}\label{eq:initialization_simple_metric}
\vy^0(x_1,x_2)~=~\big(x_1,x_2,0.8(x_1-0.5)(x_1+0.5)(x_2-0.5)(x_2+0.5)\big).
\end{equation}
Tables \ref{tab:conv_h} and \ref{tab:conv_h_2} display the results. We see that in Table \ref{tab:conv_h} both $e_h[\vy^{\infty}_h]$ and $|E_h[\vy_h]|$ are rather insensitive to $\tau$ but $N$ decreases with increasing $\tau$. The fact that performance does not improve for smaller $\tau$ 
motivates us to explore the largest admissible time step $\tau_{\max}$ with various $h$ in Table \ref{tab:conv_h_2}, which also reveals the convergence of our method.

\begin{table}[htbp]
\begin{center}
\begin{tabular}{|c|c|c|c|}
\hline
$\tau$ & $e_h[\vy^{\infty}_h]$ & $|E_h[\vy^{\infty}_h]|$ & $N$  \\
\hline
0.2 &  4.66909E-3 & 2.3484E-5 & 2304 \\ \hline
0.4 &  4.66909E-3 & 2.3484E-5 & 1151 \\ \hline
0.8 &  4.66910E-3 & 2.3484E-5 & 574 \\ \hline
1.6 &  4.66918E-3 & 2.3482E-5 & 286 \\ \hline
3.2 &  diverge & diverge & diverge \\
\hline
\end{tabular}
\vspace{0.3cm}
\caption{Experiment 1 with the blueprinted director field \eqref{eq:cont_dir_1}. This reveals the influence of $\tau$ on errors and the number of gradient flow iterations $N$ with fixed $h=1/32$.  
} \label{tab:conv_h}
\end{center}
\end{table}
\begin{table}[htbp]
\begin{center}
\begin{tabular}{|c|c|c|c|c|}
\hline
h & $\tau_{\max}$ & $e_h[\vy^{\infty}_h]$ & $|E_h[\vy^{\infty}_h]|$ & $N$  \\
\hline
1/16 & 2.23 & 9.45213E-3 & 8.7909E-5 & 267  \\ \hline
 1/32 & 2.11 & 4.66924E-3 & 2.3482E-5 & 216 \\ \hline
 1/64 &  2.10 & 2.30916E-3 & 5.7742E-6 & 130 \\ \hline
 1/128 & 2.09 & 1.22053E-3 & 1.5746E-6 & 129\\
\hline
\end{tabular}
\vspace{0.3cm}
\caption{Experiment 1 with the blueprinted director field \eqref{eq:cont_dir_1}. This gives the largest admissible time step $\tau_{\max}$ that guarantees the well-posedness and convergence of Newton step for various $h$. Convergence of errors as $h\to0$ is observed with corresponding $\tau_{\max}$.  
} \label{tab:conv_h_2}
\end{center}
\end{table}

\medskip
\emph{Experiment 2: effect of regularization}. We consider the same set-up as \emph{Experiment 1} but instead of $c_r=0$ we take $c_r=1$.

\medskip
\emph{Experiment 3: $\vm$ with defects}. 
We consider the set-up in Section \ref{sec:lc-defects}. The director field $\vm$ is the degree $3/2$ defect given in \eqref{eq:m_rot_sym_2}. The parameters are those in \eqref{eq:params_rot_sym_2}, but we take $c_r=0$ instead of $c_r=1$.

Errors for \emph{Experiments 1,2,3} are plotted in Fig. \ref{fig:conv-plot} for meshsizes
$h=1/16,1/32,$ $1/64,1/128$. We discuss them next.

\begin{figure}[htbp]
\begin{center}
\begin{tikzpicture}[scale=.75]
\begin{loglogaxis}
        [
            xlabel=$h$,
            ylabel style={rotate=-90},
            ylabel= error,
            legend style=
            {
                at={(1,1)},
                anchor=north west,
                draw=none,
                fill=none
            },
            legend cell align=left,
            grid=major,
            clip=false
        ]
\addplot table {
0.0625 9.45213E-3
0.03125 4.66924E-3
0.015625 2.30916E-3
0.0078125 1.22053E-3
};
\addplot table {
0.0625 0.01085257111159688
0.03125 0.004954169769875165
0.015625 0.002359597346030662
0.0078125 0.0012895314351707174
};
\addplot table {
0.0625 0.08063830702221442
0.03125 0.041938234260815305
0.015625 0.021253382529430286
0.0078125 0.010701233849969292
};
\addplot table {
0.0625 8.7909E-5
0.03125 2.3482E-5
0.015625 5.7742E-6
0.0078125 1.5746E-6
};
\addplot table {
0.0625 0.00011449
0.03125 0.0000264577
0.015625 6.03216E-6
0.0078125 1.73496E-6
};
\addplot table {
0.0625 0.00630405
0.03125 0.00224407
0.015625 0.000728464
0.0078125 0.000224037
};

\logLogSlopeTriangle{0.9}{0.2}{0.85}{1}{black};
\logLogSlopeTriangle{0.9}{0.2}{0.29}{2}{black};
\logLogSlopeTriangle{0.9}{0.2}{0.63}{2}{black};
\legend
            {
                Experiment 1: $e_h[\vy^{\infty}_h]$, 
                Experiment 2: $e_h[\vy^{\infty}_h]$,
                Experiment 3: $e_h[\vy^{\infty}_h]$,
                Experiment 1: $|E_h[\vy^{\infty}_h]|$, 
                Experiment 2: $|E_h[\vy^{\infty}_h]|$,
                Experiment 3: $|E_h[\vy^{\infty}_h]|$
            }
\end{loglogaxis}
\end{tikzpicture}
\vskip-0.3cm
\caption{Convergence of errors for \emph{Experiments 1,2,3}. We can see that the regularization has almost no influence on convergence rates, while it results in a slightly larger value of errors. For \emph{Experiment 3} with discontinuous $\vm$ the errors are significantly larger. In all cases we observe that $e_h[\vy_h^{\infty}]$ is linear in $h$, while $|E_h[\vy^{\infty}_h]|$ is quadratic in $h$ for \emph{Experiments 1,2} and has a rate slightly worse than quadratic (it is approximately $\mathcal{O}(h^{\log_23})$) for \emph{Experiment 3}.}
\end{center}
\label{fig:conv-plot}
\end{figure}
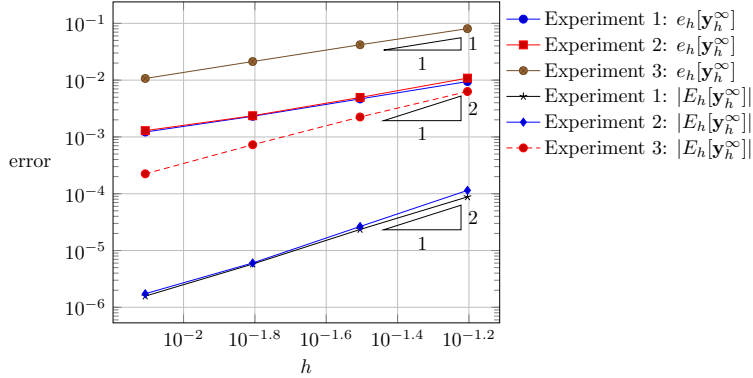

\subsection{Conclusions}
We conclude with a summary of quantitative observations.
\begin{itemize}
\item The metric deviation $e_h[\vy_h^{\infty}]$ converges as $\mathcal{O}(h)$. The energy error $|E_h[\vy^{\infty}_h]|$ converges as $\mathcal{O}(h^2)$ or sub-quadratically, depending on the regularity of $\vm$.
\begin{itemize}
\item $|E_h[\vy^{\infty}_h]|$ converges as $\mathcal{O}(h^2)$ in Experiments 1 and 2, when $\vm$ is smooth and $g$ is likely to admit a $H^2$ isometric immersion. This computational result corroborates the validity of Assumption \ref{as:H2-immersibility} and the energy scaling in Proposition \ref{prop:energy-scaling}.
\item $|E_h[\vy^{\infty}_h]|$ converges sub-quadratically in Experiment 3, when $\vm$ has a degree $3/2$ defect. It is plausible that $g$ does not admit a $H^2$ isometric immersion, and if so the validity of Assumption \ref{as:H2-immersibility} is questionable. It is worth realizing that this assumption is responsible for the quadratic energy scaling in Proposition \ref{prop:energy-scaling}.
\end{itemize}
\item The Newton sub-iteration is well-posed and convergent when $\tau$ is small enough. The influence of $h$ on $\tau_{\max}$ is negligible.     
\item Once $\tau$ is chosen so that the Newton method is well-posed and convergent, further decreasing of $\tau$ has only a negligible influence on errors. 
\item For fixed $h$, the number of gradient flow iterations $N=\mathcal{O}(\tau^{-1})$, and so does the computational time. 
\end{itemize}
These conclusions indicate the convergence of the method and the fact that an ideal choice of $\tau$ is its largest admissible value $\tau_{\max}$ for various problems. We do not need to take $\tau\to0$ as meshes refine, and $\tau_{\max}$ provides a moderate upper bound for $\tau$. This is an advantage compared to a linearized gradient flow (e.g. \cite{bonito2023numerical}) and a fixed point sub-iteration scheme (e.g. \cite{bartels2017bilayer}) in that both require $\tau$ depending on $h$.

\section*{Acknowledgements}
Lucas Bouck was supported by the NSF grant DGE-1840340. Ricardo H. Nochetto and Shuo Yang were partially supported by NSF grant DMS-1908267.

\bibliographystyle{siamplain}
\bibliography{LCE_ref_NA}
\end{document}